\documentclass[a4paper,12pt,reqno]{amsproc}

\usepackage{amsfonts,amsmath,hyperref,color}
\usepackage{amsthm,amssymb,bbm}
\allowdisplaybreaks

\usepackage[T1]{fontenc}
\usepackage[utf8]{inputenc}
\usepackage[a4paper,margin=1in]{geometry}
\usepackage{enumerate}

\renewcommand{\leq}{\leqslant}
\renewcommand{\geq}{\geqslant}
\renewcommand{\le}{\leqslant}
\renewcommand{\ge}{\geqslant}
\renewcommand{\Re}{\ensuremath{\operatorname{Re}}}
\renewcommand{\Im}{\ensuremath{\operatorname{Im}}}

\newcommand{\esssup}{\operatorname*{ess\,sup}}

\renewcommand{\phi}{\varphi}
\renewcommand{\epsilon}{\varepsilon}

\newcommand{\real}{\mathbb{R}}
\newcommand{\comp}{\mathbb{C}}
\newcommand{\nat}{\mathbb{N}}
\newcommand{\Ss}{\mathbb{S}}

\theoremstyle{plain}
\newtheorem{lemma}{Lemma}[section]
\newtheorem{theorem}[lemma]{Theorem}
\newtheorem{corollary}[lemma]{Corollary}

\theoremstyle{definition}

\newtheorem{remark}[lemma]{Remark}
\newtheorem*{examples}{Examples}
\newtheorem*{ack}{Acknowledgements}

\usepackage{marginnote}
\begin{document}
\begin{flushright}\small
\underline{\emph{To appear in: Journal of Spectral Theory }}
\bigskip\bigskip
\end{flushright}
\title[Liouville theorem for sub-exponentially growing solutions]{\bfseries An extension of the Liouville theorem
for Fourier multipliers to sub-exponentially growing solutions}
\dedicatory{\bfseries To Brian Davies on the occasion of his 80th birthday}
\author[D.~Berger]{David Berger}
\author[R.L.~Schilling]{Ren\'e L.\ Schilling}
\address[D.~Berger \& R.L.\ Schilling]{TU Dresden\\ Fakult\"{a}t Mathematik\\ Institut f\"{u}r Mathematische Stochastik\\ 01062 Dresden, Germany}
\email{david.berger2@tu-dresden.de}
\email{rene.schilling@tu-dresden.de}

\author[E.~Shargorodsky]{Eugene Shargorodsky}
\address[E.~Shargorodsky]{King's College London\\ Department of Mathematics\\ Strand, London, WC2R 2LS, UK}
\email{eugene.shargorodsky@kcl.ac.uk}

\author[T.~Sharia]{Teo Sharia}
\address[T.~Sharia]{Royal Holloway University of London\\ Department of Mathematics\\ Egham, Surrey, TW20 0EX, UK}
\email{t.sharia@rhul.ac.uk}

\begin{abstract}
    We study the equation $m(D)f=0$ in a large class of sub-exponentially growing functions.
    Under appropriate restrictions on $m \in C(\real^n)$ we show that every such solution can be analytically continued to a sub-exponentially growing entire function on $\comp^n$ if, and only if,  $m(\xi) \neq 0$ for $\xi \neq 0$.
\end{abstract}


\maketitle

\section{Introduction}

The classical Liouville theorem for the Laplace operator $\Delta := \sum_{k=1}^n \frac{\partial^2}{\partial x_k^2}$ on $\real^n$ says that every bounded (polynomially bounded) solution of the equation $\Delta f = 0$ is in fact constant (is a polynomial). Recently, similar results have been obtained for solutions of more general equations of the form $m(D)f = 0$, where $m(D) :=  \mathcal{F}^{-1} m(\xi) \mathcal{F}$, and
\begin{gather*}
    \mathcal{F} \phi(\xi)
    = \widehat\phi(\xi)
    = \int_{\real^n} e^{-ix\cdot\xi}\phi(x)\,dx
    \quad\text{and}\quad
    \mathcal{F}^{-1} u(x)
    = (2\pi)^{-n} \int_{\real^n} e^{ix\cdot\xi} u(\xi)\,d\xi
\end{gather*}
are the Fourier and the inverse Fourier transforms, see \cite{Ali-et-al20,BS22,BSS23,GK23}, and the references therein. Namely, it was shown that, under appropriate restrictions on $m \in C(\real^n)$, the implication
\begin{gather*}
    f  \text{\ \ is bounded (polynomially bounded) and\ \ }  m(D)f = 0   \\
    \implies \quad f \text{\ \ is constant (is a polynomial)}
\end{gather*}
holds if, and only if,  $m(\xi) \neq  0$ for $\xi \neq  0$. Much of this research has been motivated by applications to infinitesimal generators of L\'evy processes.

In this paper, we study solutions of $m(D)f = 0$ that can grow faster than  any polynomial.  Of course, one cannot expect such solutions to have a simple structure,  not  even in the case of $\Delta f = 0$ in $\real^2$, see, e.g., \cite[Ch.~I, \S~2]{L64}. We consider sub-exponentially growing solutions whose growth is controlled by a submultiplicative function, cf.\ \eqref{submult}, satisfying the Beurling--Domar condition \eqref{B-D}, and we show that, under appropriate restrictions on $m \in C(\real^n)$, every such solution admits analytic continuation to a sub-exponentially growing entire function on $\comp^n$ if, and only if,  $m(\xi) \neq  0$ for $\xi \neq  0$, see Corollary \ref{th:3.1}. Results of this type have been obtained for solutions of partial differential equations with constant coefficients by A. Kaneko and G.E. \v{S}ilov, see \cite{Kan98,Kan00,S61}, \cite[Ch.~10, Sect.~2, Theorem~2]{F63}, and Section \ref{Rems} below.

Keeping in mind applications to infinitesimal generators of L\'evy processes, we do not assume that $m$ is the Fourier transform of a distribution with compact support, so our setting is different from that in, e.g., \cite{E60}, \cite[Ch.~XVI]{H83_2}.

The paper is organized as follows. In Chapter \ref{Ch2}, we consider submultiplicative functions satisfying the Beurling--Domar condition. For every such function $g$, we introduce an auxiliary function $S_g$, see \eqref{Sg}, \eqref{Jnfty}, which appears in our main estimates. Chapter \ref{Ch3} contains weighted $L^p$ estimates for entire functions on $\comp^n$, which are a key ingredient in the proof of our main results in Chapter \ref{Ch4}. Another key ingredient is the Tauberian theorem \ref{Taub}, which is similar to \cite[Thm.~7]{BSS23} and \cite[Thm.~9.3]{Ru73}. The main difference is that the function $f$ in Theorem \ref{Taub} is not assumed to be polynomially bounded, and hence it might not be a tempered distribution. So, we avoid using the Fourier transform $\widehat{f} = \mathcal{F} f$ and its support (and non-quasianalytic type ultradistributions). Although we are mainly interested in the case $m(\xi) \neq  0$ for $\xi \neq  0$, we also prove a Liouville type result for $m$ with compact zero set $\left\{\xi \in \real^n \mid  m(\xi) = 0\right\}$, see Theorem \ref{compZ}. Finally, we discuss in Section \ref{Rems} A.~Kaneko's Liouville type results for partial differential equations with constant coefficients, cf.~\cite{Kan98,Kan00}, which show that the Beurling--Domar condition is in a sense optimal in our setting.

\section{Submultiplicative functions and the Beurling--Domar condition}\label{Ch2}

Let $g : \real^n\to (0, \infty)$ be a locally bounded, measurable \emph{submultiplicative} function,  i.e.\ a locally bounded measurable function satisfying
\begin{gather*}
    g(x + y) \le C g(x)g(y) \quad\text{for all}\quad x, y \in \real^n,
\end{gather*}
where the constant $C \in [1, \infty)$ does not depend on $x$ and $y$. Without loss of generality, we will always assume that $g\geq 1$, as otherwise we can replace $g$ with $g+1$. Also, replacing $g$ with $Cg$, we can assume that
\begin{equation}\label{submult}
    g(x + y) \le  g(x)g(y) \quad\text{for all}\quad x, y \in \real^n .
\end{equation}
A locally bounded submultiplicative function is exponentially bounded, i.e.
\begin{equation}\label{expbdd}
    |g(x)| \le C e^{a|x|}
\end{equation}
for suitable constants $C, a > 0$, see \cite[Section~25]{sato} or \cite[Ch.~VII]{HP57}.

We will say that $g$ satisfies the \emph{Beurling--Domar} condition if
\begin{equation}\label{B-D}
    \sum_{l = 1}^\infty \frac{\log{g(l x)}}{l^2} < \infty \quad\text{for all}\quad x \in \real^n .
\end{equation}
If $g$ satisfies the Beurling--Domar condition, then it also satisfies the Gelfand--Raikov--Shilov condition
\begin{gather*}
    \lim_{l \to \infty} g(lx)^{1/l} = 1 \quad\text{for all}\quad x \in \real^n ,
\end{gather*}
while $g(x) = e^{|x|/\log(e + |x|)}$ satisfies the latter but not the former condition, see \cite{G07}. It is also easy to see that $g(x) = e^{|x|/\log^\gamma(e + |x|)}$ satisfies the Beurling--Domar condition if, and only if,  $\gamma > 1$. The function
\begin{gather*}
    g(x) = e^{a |x|^b} (1 + |x|)^s (\log(e + |x|))^t
\end{gather*}
satisfies the Beurling--Domar condition for any $a, s, t \ge 0$ and $b \in [0, 1)$, see \cite{G07}.

\begin{lemma}\label{unif}
    Let $g : \real^n\to [1, \infty)$ be a locally bounded, measurable submultiplicative function satisfying the Beurling--Domar condition \eqref{B-D}. Then for every $\varepsilon > 0$, there exists $R_\varepsilon \in (0, \infty)$ such that
    \begin{equation}\label{uest}
        \int_{R_\varepsilon}^\infty \frac{\log{g(\tau x)}}{\tau^2}\, d\tau < \varepsilon
        \quad\text{for all}\quad x \in \Ss^{n - 1} := \left\{y \in \real^n : \ |y| = 1\right\}.
\end{equation}
\end{lemma}
\begin{proof}
Since $g \ge 1$ is locally bounded,
\begin{equation}\label{M}
    0 \le M := \sup_{|y| \le 1} \log{g(y)} < \infty .
\end{equation}
Take any $x \in \Ss^{n - 1}$. It follows from \eqref{submult} that
\begin{gather*}
    \log{g((l + 1) x)} - M
    \le \log{g(\tau x)}
    \le \log{g(l x)} + M \quad\text{for all}\quad  \tau \in [l, l+1] .
\end{gather*}
Hence,
\begin{gather*}
    \sum_{l = L}^\infty \frac{\log{g((l + 1) x)} - M}{(l + 1)^2}
    \le \sum_{l = L}^\infty \int_{l}^{l + 1} \frac{\log{g(\tau x)}}{\tau^2}\, d\tau
    \le \sum_{l = L}^\infty \frac{\log{g(l x)} + M}{l^2},
\end{gather*}
and this implies for all $L\in\nat$ that
\begin{gather}\label{intsum}
    \sum_{l = L + 1}^\infty \frac{\log{g(l x)}}{l^2} - \frac{M}{L}
    \le \int_{L}^\infty \frac{\log{g(\tau x)}}{\tau^2}\, d\tau
    \le \sum_{l = L}^\infty \frac{\log{g(l x)}}{l^2} + \frac{M}{L - 1}.
\end{gather}

Let
\begin{equation}\label{ej}\begin{gathered}
    \mathbf{e}_j := (\underbrace{0,\dots,0}_{j - 1},1,0,\dots, 0) , \ j = 1, \dots, n, \qquad
    \mathbf{e}_0 := \frac 1{\sqrt n} \left(1, \dots, 1\right), \\
    Q := \left\{y = (y_1, \dots, y_n) \in \real^n : \ \frac{1}{2\sqrt{n}} < y_j <  \frac{2}{\sqrt{n}}, \, j = 1, \dots, n\right\} .
\end{gathered}\end{equation}
For every $x \in \Ss^{n - 1}$ there exists an orthogonal matrix $A_x \in O(n)$ such that $x = A_x \mathbf{e}_0$. Hence $\{AQ\}_{A \in O(n)}$ is an open cover of $\Ss^{n - 1}$. Let $\{A_kQ\}_{k = 1, \dots, K}$ be a finite subcover. Take an arbitrary $\varepsilon > 0$. It follows from \eqref{B-D} and \eqref{intsum} that there exists some $R_\varepsilon > 0$ for which
\begin{gather*}
    \int_{\frac{R_\varepsilon}{2\sqrt{n}}}^\infty \frac{\log{g(\tau A_k \mathbf{e}_j)}}{\tau^2}\, d\tau
    < \frac{\varepsilon}{2\sqrt{n}},
    \quad k = 1, \dots, K, \quad j = 1, \dots, n .
\end{gather*}
For any $x \in \Ss^{n - 1}$, there exist $k = 1, \dots, K$ and $a_j \in \left(\frac{1}{2\sqrt{n}}, \frac{2}{\sqrt{n}}\right)$, $j = 1, \dots, n$ such that
\begin{gather*}
    x = \sum_{j = 1}^n a_j A_k \mathbf{e}_j .
\end{gather*}
Using \eqref{submult}, one gets
\begin{align*}
    \int_{R_\varepsilon}^\infty \frac{\log{g(\tau x)}}{\tau^2}\, d\tau
    &\le \sum_{j = 1}^n \int_{R_\varepsilon}^\infty \frac{\log{g(\tau a_j A_k \mathbf{e}_j)}}{\tau^2}\, d\tau\\
    &=  \sum_{j = 1}^n a_j \int_{a_j R_\varepsilon}^\infty \frac{\log{g(r A_k \mathbf{e}_j)}}{r^2}\, dr \\
    &\le \sum_{j = 1}^n \frac{2}{\sqrt{n}} \int_{\frac{R_\varepsilon}{2\sqrt{n}}}^\infty \frac{\log{g(r A_k \mathbf{e}_j)}}{r^2}\, dr\\
    &< \sum_{j = 1}^n \frac{2}{\sqrt{n}}\cdot \frac{\varepsilon}{2\sqrt{n}} = n\, \frac{\varepsilon}{n} = \varepsilon .
\qedhere
\end{align*}
\end{proof}

Let
\begin{align*}
& I_{g, x}(r) := \int_{\max\{r, 1\}}^\infty \frac{\log{g(\tau x)}}{\tau^2}\, d\tau < \infty , \\
& J_{g, x}(r) := \frac{1}{\max\{r, 1\}^2}\int_0^r \log{g(\tau x)}\, d\tau < \infty , \\
& S_{g, x}(r) := \frac{1}{\pi} \int_{-\infty}^\infty \frac{\log{g(\tau x)}}{\tau^2 + \max\{r, 1\}^2}\, d\tau
\quad r \ge 0, \
x \in \Ss^{n - 1} .
\end{align*}
One has, for $r > 1$ and any $\beta \in (0, 1)$,
\begin{align}
    J_{g, x}(r)
    &=  \frac{1}{r^2}\int_0^{r} \log{g(\tau x)}\, d\tau \notag\\
    &= \frac{1}{r^2}\int_0^1 \log{g(\tau x)}\, d\tau
        + \frac{1}{r^{2(1 - \beta)}}\int_1^{r^\beta}\frac{\log{g(\tau x)}}{r^{2\beta}}\, d\tau
        + \int_{r^\beta}^{r} \frac{\log{g(\tau x)}}{r^2}\, d\tau \notag \\
    &\le \frac{M}{r^2} + \frac{1}{r^{2(1 - \beta)}}\int_1^{r^\beta}\frac{\log{g(\tau x)}}{\tau^2}\, d\tau
        + \int_{r^\beta}^{r} \frac{\log{g(\tau x)}}{\tau^2}\, d\tau  \notag \\
    &\le  \frac{M}{r^2} + \frac{I_{g, x}(1)}{r^{2(1 - \beta)}} + I_{g, x}(r^{\beta}),  \label{JIx}
\end{align}
see \eqref{M}. Further, if $r > 1$, then
\begin{align}
    \pi S_{g, x}(r)
    &= \int_{0}^\infty \frac{\log{g(\tau x)}}{\tau^2 + r^2}\, d\tau
        +  \int_{0}^\infty \frac{\log{g(-\tau x)}}{\tau^2 + r^2}\, d\tau \notag \\
    &\le  \int_0^{r} \frac{\log{g(\tau x)}}{r^2}\, d\tau
        + \int_{r}^\infty \frac{\log{g(\tau x)}}{\tau^2}\, d\tau
        +  \int_0^{r} \frac{\log{g(-\tau x)}}{r^2}\, d\tau
        + \int_{r}^\infty \frac{\log{g(-\tau x)}}{\tau^2}\, d\tau \notag \\
    &=   I_{g, x}(r) + J_{g, x}(r) + I_{g, -x}(r) + J_{g, -x}(r), \label{upperx}
\intertext{and, with a similar calculation,}
    \pi S_{g, x}(r)
    & \ge   \int_0^{r} \frac{\log{g(\tau x)}}{2r^2}\, d\tau
            + \int_{r}^\infty \frac{\log{g(\tau x)}}{2\tau^2}\, d\tau
            +  \int_0^{r} \frac{\log{g(-\tau x)}}{2r^2}\, d\tau
            + \int_{r}^\infty \frac{\log{g(-\tau x)}}{2\tau^2}\, d\tau \notag \\
    &=   \frac12\left(I_{g, x}(r) + J_{g, x}(r) + I_{g, -x}(r) + J_{g, -x}(r)\right). \label{lowerx}
\end{align}

Since $g$ is locally bounded, it follows from Lemma \ref{unif} that $I_g$ defined by
\begin{equation}\label{Ig}
    I_g(r) := \sup_{x \in \Ss^{n - 1}} I_{g, x}(r)
    = \sup_{x \in \Ss^{n - 1}}\int_{\max\{r, 1\}}^\infty \frac{\log{g(\tau x)}}{\tau^2}\, d\tau < \infty ,
\end{equation}
is a decreasing function such that
\begin{equation}\label{Infty}
    I_g(r) \to 0 \quad\text{as}\quad r \to \infty .
\end{equation}
Let
\begin{align}
    J_g(r) &:= \sup_{x \in \Ss^{n - 1}} J_{g, x}(r) = \sup_{x \in \Ss^{n - 1}} \frac{1}{\max\{r, 1\}^2}\int_0^r \log{g(\tau x)}\, d\tau , \label{Jg}  \\
    S_g(r) &:= \sup_{x \in \Ss^{n - 1}} S_{g, x}(r) = \sup_{x \in \Ss^{n - 1}} \frac{1}{\pi} \int_{-\infty}^\infty \frac{\log{g(\tau x)}}{\tau^2 + \max\{r, 1\}^2}\, d\tau . \label{Sg}
\end{align}
Then, in view of \eqref{JIx}, \eqref{upperx}, \eqref{lowerx},
\begin{gather*}
    J_{g}(r) \le  \frac{M}{r^2} + \frac{I_{g}(1)}{r^{2(1 - \beta)}} + I_{g}(r^{\beta}) ,
    \\
    \frac1{2\pi}\, \max\left\{I_g(r), J_g(r)\right\} \le S_g(r) \le \frac{2}{\pi}\left(I_g(r) + J_g(r)\right).
\end{gather*}
Thus, $J_g(r) \to 0$, and
\begin{equation}\label{Jnfty}
    S_g(r) \to 0 \quad\text{as}\quad r \to \infty,
\end{equation}
see \eqref{Infty}. It is clear that
\begin{equation}\label{Sgdecr}
    S_g(r) = S_g(1) \text{\ \ for\ \ } r \in [0, 1], \quad\text{and}\quad S_g \text{\ \ is a decreasing function.}
\end{equation}

\begin{examples}
\begin{enumerate}[1)]
\item
If $g(x)  = (1 + |x|)^s$, $s \ge 0$, then we have for all $r\geq 1$
\begin{align}
    S_g(r)
    &= \frac1{\pi}\int_{-\infty}^\infty \frac{s\log{(1 + |\tau|)}}{\tau^2 + r^2}\, d\tau \\
    &= \frac{s}{\pi r} \int_{-\infty}^\infty \frac{\log{(1 + r|\lambda|)}}{\lambda^2 + 1}\, d\lambda \notag \\
    &\le \frac{s}{\pi r} \int_{-\infty}^\infty \frac{\log{(1 + |\lambda|)}}{\lambda^2 + 1}\, d\lambda + \frac{s \log{(1 + r)}}{\pi r} \int_{-\infty}^\infty \frac{1}{\lambda^2 + 1}\, d\lambda \notag \\
    &= \frac{c_1 s}{r} + \frac{s \log{(1 + r)}}{r},  \label{s}
\end{align}
where
\begin{gather*}
    c_1 := \frac{1}{\pi} \int_{-\infty}^\infty \frac{\log{(1 + |\lambda|)}}{\lambda^2 + 1}\, d\lambda < \infty .
\end{gather*}

\item
If $g(x)  = (\log(e + |x|))^t$, $t \ge 0$, then using the obvious inequality
\begin{gather*}
    u + v \le 2uv , \qquad u, v \ge 1,
\end{gather*}
yields for $r\geq 1$
\begin{align}
    S_g(r)
    &= \frac1{\pi}\int_{-\infty}^\infty \frac{t\log\log{(e + |\tau|)}}{\tau^2 + r^2}\, d\tau \notag\\
    &= \frac{t}{\pi r} \int_{-\infty}^\infty \frac{\log\log{(e + r|\lambda|)}}{\lambda^2 + 1}\, d\lambda \notag \\
    &\le \frac{t}{\pi r} \int_{-\infty}^\infty \frac{\log\big(\log{(e + |\lambda|)} + \log{(e + r)}\big)}{\lambda^2 + 1}\, d\lambda \notag \\
    &\le \frac{t}{\pi r} \int_{-\infty}^\infty \frac{\log\big(2\log{(e + |\lambda|)\big)}}{\lambda^2 + 1}\, d\lambda + \frac{t \log\log{(e + r)}}{\pi r} \int_{-\infty}^\infty \frac{1}{\lambda^2 + 1}\, d\lambda \notag \\
    &= \frac{c_2 t}{r} + \frac{t \log\log{(e + r)}}{r}, \label{t}
\end{align}
where
\begin{gather*}
    c_2 := \frac{1}{\pi} \int_{-\infty}^\infty \frac{\log\big(2\log{(e + |\lambda|)\big)}}{\lambda^2 + 1}\, d\lambda < \infty .
\end{gather*}

\item
If $g(x)  = e^{a |x|^b}$, $a \ge 0$, $b \in [0, 1)$,  then we have for all $r\geq 1$
\begin{align}
    S_g(r)
    &= \frac1{\pi}\int_{-\infty}^\infty \frac{a|\tau|^b}{\tau^2 + r^2}\, d\tau \notag\\
    &= \frac{a r^{b - 1}}{\pi} \int_{-\infty}^\infty \frac{|\lambda|^b}{\lambda^2 + 1}\, d\lambda\notag\\
    &= \frac{2a r^{b - 1}}{\pi} \int_{0}^\infty \frac{t^b}{t^2 + 1}\, dt \notag \\
    &= \frac{a r^{b - 1}}{\pi} \int_{0}^\infty \frac{s^\frac{b - 1}{2}}{s + 1}\, ds\notag\\
    &= \frac{a r^{b - 1}}{\sin\left(\frac{1 - b}{2} \pi\right)},  \label{ab}
\end{align}
see, e.g.\ \cite[Ch.~V, Example~2.12]{C78}.

\item
Finally, let $g(x) = e^{|x|/\log^\gamma(e + |x|)}$, $\gamma > 1$. Since
\begin{gather*}
    \frac{\tau(e + \tau)}{\tau^2 + r^2}
    = \frac{1 +\frac{e}{\tau}}{1 + \frac{r^2}{\tau^2}}
    \le 1 +\frac{e}{\tau}
    \le 1 +\frac{e}{r}  \quad\text{for}\quad \tau \ge r ,
\end{gather*}
then for any $\beta \in (0, 1)$  and $r\geq 1$
\begin{align*}
    S_g(r)
    &= \frac1{\pi}\int_{-\infty}^\infty \frac{|\tau|}{(\tau^2 + r^2) \log^\gamma(e + |\tau|)}\, d\tau\\
    &= \frac2{\pi}\int_{0}^\infty \frac{\tau}{(\tau^2 + r^2) \log^\gamma(e + \tau)}\, d\tau \\
    &= \frac2{\pi}  \left(\int_{0}^{r^\beta} + \int_ {r^\beta}^r + \int_{r}^\infty\right)  \frac{\tau}{(\tau^2 + r^2) \log^\gamma(e + \tau)}\, d\tau  \\
    &\le  \frac2{\pi}\int_{0}^{r^\beta}  \frac{\tau}{\tau^2 + r^2}\, d\tau
        + \frac2{\pi \log^\gamma(e + r^\beta)}\int_{r^\beta}^r \frac{\tau}{\tau^2 + r^2}\, d\tau  \\
    &\quad\mbox{} + \frac2{\pi} \left(1 +\frac{e}{r}\right) \int_{r}^\infty \frac{1}{(e + \tau) \log^\gamma(e + \tau)}\, d\tau  \\
    &= \frac1{\pi} \log(\tau^2 + r^2)\big|_0^{r^\beta} + \frac1{\pi \log^\gamma(e + r^\beta)} \log(\tau^2 + r^2)\big|_{r^\beta}^r  \\
    &\quad\mbox{} + \frac2{\pi} \left(1 +\frac{e}{r}\right) \frac{1}{1 - \gamma} \log^{1 - \gamma}(e + \tau)\big|_{r}^\infty    \\
    &\le  \frac1{\pi} \log(1 + r^{2(\beta - 1)}) + \frac{\log 2}{\pi \log^\gamma(e + r^\beta)}
        + \frac2{\pi} \left(1 +\frac{e}{r}\right) \frac{1}{\gamma - 1} \log^{1 - \gamma}(e + r)  \\
    &\le \frac{r^{2(\beta - 1)}}{\pi} + \frac{\log 2}{\pi \log^\gamma(e + r^\beta)}
        + \frac2{\pi} \left(1 +\frac{e}{r}\right) \frac{1}{\gamma - 1} \log^{1 - \gamma}(e + r).
\end{align*}
Since
\begin{gather*}
    \lim_{r \to \infty} \frac{r^{2(\beta - 1)} + (\log{2}) \log^{-\gamma}(e + r^\beta)}{\log^{-\gamma}(e + r)}
    = \frac{\log 2}{\beta^\gamma} \quad\text{for all}\quad \beta \in (0, 1) ,
\end{gather*}
one gets, if we take $\beta \in \left((\log 2)^{1/\gamma}, 1\right)$, the following estimate
\begin{equation}\label{gamma}
    S_g(r)  \le \frac{\log^{-\gamma}(e + r)}{\pi} + \frac2{\pi} \left(1 +\frac{e}{r}\right) \frac{1}{\gamma - 1} \log^{1 - \gamma}(e + r)
\end{equation}
for sufficiently large $r$.
\end{enumerate}
\end{examples}

\section{Estimates for entire functions}\label{Ch3}

Let $1\leq p\leq \infty$ and let $\omega : \real^n\to [0, \infty)$ be a measurable function such that $\omega > 0$ Lebesgue almost everywhere. We set
\begin{equation}\label{Lgp}
    \|f\|_{L^p_{\omega}}:= \|\omega f\|_{L^p}
    \quad\text{and}\quad
    L_\omega^p(\real^n)
    :=\left\{f : \real^n \to \comp \mid f \text{\ measurable,\ } \|f\|_{L^p_{\omega}} < \infty\right\}.
\end{equation}

\begin{lemma}\label{x1}
Let $g : \real^n\to [1, \infty)$ be a locally bounded, measurable submultiplicative function satisfying the Beurling--Domar condition \eqref{B-D}. Let $\varphi$ be a measurable function such that for almost every $x' = (x_2, \dots, x_n) \in \real^{n - 1}$, $\varphi(z_1, x')$ is analytic in $z_1$ for $\Im  z_1 > 0$ and continuous up to $\real$. Suppose also that $\log{|\varphi(z_1, x')|} = O(|z_1|)$ for $|z_1|$ large, $\Im z_1 \ge 0$, and that the restriction of $\varphi$ to $\real^n$ belongs to $L_{g^{\pm 1}}^p(\real^n)$, $1 \le p \le \infty$. Finally, suppose that
\begin{equation}\label{k}
    k_\varphi := \esssup_{x' \in \real^{n - 1}}\left(\limsup_{0 < y_1 \to \infty} \frac{\log{|\varphi(iy_1, x')|}}{y_1}\right) < \infty.
\end{equation}
Then
\begin{equation}\label{est}
    \|\varphi(\cdot + iy_1, \cdot)\|_{L_{g^{\pm 1}}^p(\real^n)} \le C_g e^{\left(k_\varphi  + S_g(y_1)\right) y_1} \|\varphi\|_{L_{g^{\pm 1}}^p(\real^n)}, \quad  y_1 > 0,
\end{equation}
see \eqref{Sg}, \eqref{Jnfty}, where the constant $C_g < \infty$ depends only on $g$.
\end{lemma}
\begin{proof}
    Let $a^+ := \max\{a, 0\}$ for $a \in \real$. It follows from \eqref{submult} that
\begin{align*}
    \int_{-\infty}^\infty \frac{\log^+{\left(g^{\mp 1}(t, x')\right)}}{1 + t^2}\, dt
    &\le \int_{-\infty}^\infty \frac{\log{\left(g(t, x')\right)}}{1 + t^2}\, dt \\
    &\le \int_{-\infty}^\infty \frac{\log(g(t, 0)) + \log(g(0, x'))}{1 + t^2}\, dt\\
    &\le \pi \left((S_g(1) + \log(g(0, x'))\right) < +\infty .
\end{align*}
Since $g^{\pm 1}\varphi \in L^p(\real^n)$,  Fubini's theorem implies that
\begin{gather*}
    g^{\pm 1}(\cdot, x')\varphi(\cdot, x') \in L^p(\real)
\end{gather*}
for  Lebesgue  almost all $x' \in \real^{n - 1}$. For such $x' \in \real^{n - 1}$,
\begin{align*}
    &\int_{-\infty}^\infty \frac{\log^+{|\varphi(t, x')|}}{1 + t^2}\, dt \\
    &\quad \le \int_{-\infty}^\infty \frac{\log^+{\left(g^{\pm 1}(t, x')|\varphi(t, x')|\right)}}{1 + t^2}\, dt
        + \int_{-\infty}^\infty \frac{\log^+{\left(g^{\mp 1}(t, x')\right)}}{1 + t^2}\, dt
        < \infty .
\end{align*}
Then
\begin{gather*}
    \log{|\varphi(x_1 + iy_1, x')|}
    \le k_\varphi y_1 + \frac{y_1}{\pi} \int_{-\infty}^\infty \frac{\log{|\varphi(t, x')|}}{(t - x_1)^2 + y_1^2}\, dt,
    \quad  x_1 \in \real, \; y_1 > 0,
\end{gather*}
cf.~\cite[Ch.~III, G, 2]{K98}, see also \cite[Ch.~V, Theorems 5 and 7]{L64}.

Applying \eqref{submult} again, one gets
\begin{align*}
    \log{g(x)} &\le \log{g(t,x')} + \log{g(x_1 - t, 0)} , \\
    \log{g(t, x')} &\le \log{g(x)} + \log{g(t - x_1, 0)}
    \quad\text{for all}\quad x = (x_1, x') \in \real^n, \; t \in \real .
\end{align*}
The latter inequality can be rewritten as follows
\begin{gather*}
    \log{g^{-1}(x)}
    \le \log{g^{-1}(t,x')} +  \log{g(t - x_1, 0)} .
\end{gather*}
Hence,
\begin{gather*}
    \log{g^{\pm1}(x)} \le \log{g^{\pm 1}(t,x')} +  \log{g(\pm(x_1 - t), 0)}
    \quad\text{for all}\quad x = (x_1, x') \in \real^n, \ t \in \real ,
\end{gather*}
and
\begin{align*}
    &\log{\big(|\varphi(x_1 + iy_1, x')|g^{\pm1}(x)\big)} \\
    &\quad\le k_\varphi y_1 + \frac{y_1}{\pi} \int_{-\infty}^\infty \frac{\log{|\varphi(t, x')|}}{(t - x_1)^2 + y_1^2}\, dt +  \log{g^{\pm1}(x)} \\
    &\quad  = k_\varphi y_1 + \frac{y_1}{\pi} \int_{-\infty}^\infty \frac{\log{|\varphi(t, x')|} + \log{g^{\pm1}(x)}}{(t - x_1)^2 + y_1^2}\, dt \\
    &\quad  \le k_\varphi y_1 + \frac{y_1}{\pi} \int_{-\infty}^\infty \frac{\log{\big(|\varphi(t, x')|g^{\pm1}(t, x')\big)}}{(t - x_1)^2 + y_1^2}\, dt
        + \frac{y_1}{\pi} \int_{-\infty}^\infty \frac{\log{g(\pm(x_1 - t), 0)}}{(t - x_1)^2 + y_1^2}\, dt \\
    &\quad  = k_\varphi y_1 + \frac{y_1}{\pi} \int_{-\infty}^\infty \frac{\log{\big(|\varphi(t, x')|g^{\pm1}(t, x')\big)}}{(t - x_1)^2 + y_1^2}\, dt
        + \frac{y_1}{\pi} \int_{-\infty}^\infty \frac{\log{g(\tau, 0)}}{\tau^2 + y_1^2}\, d\tau .
\end{align*}

If $y_1 \in [0, 1]$, then
\begin{align}
    \frac{y_1}{\pi} \int_0^\infty \frac{\log{g(\tau, 0)}}{\tau^2 + y_1^2}\, d\tau
    &\le M\,\frac{y_1}{\pi} \int_0^1 \frac{1}{\tau^2 + y_1^2}\, d\tau + \frac{y_1}{\pi} \int_1^\infty \frac{\log{g(\tau, 0)}}{\tau^2 + y_1^2}\, d\tau \notag \\
    &\le M\,\frac{y_1}{\pi} \int_{\real} \frac{1}{\tau^2 + y_1^2}\, d\tau + \frac{1}{\pi} \int_1^\infty \frac{\log{g(\tau, 0)}}{\tau^2}\, d\tau \notag\\
    &\le M + \frac{I_g(1)}{\pi}. \label{y10}
\end{align}
It follows from \eqref{Sg} that for $y_1 > 1$,
\begin{gather*}
    \frac{y_1}{\pi} \int_{-\infty}^\infty \frac{\log{g(\tau, 0)}}{\tau^2 + y_1^2}\, d\tau
    \le y_1 S_g(y_1) .
\end{gather*}
So,
\begin{align*}
    \log{\big(|\varphi(x_1 + iy_1, x')|g^{\pm1}(x)\big)} \le c_g
    &\mbox{}+ \left(k_\varphi + S_g(y_1)\right) y_1 \\
    &\mbox{}+  \frac{y_1}{\pi} \int_{-\infty}^\infty \frac{\log{\big(|\varphi(t, x')|g^{\pm1}(t, x')\big)}}{(t - x_1)^2 + y_1^2}\, dt ,
\end{align*}
where $c_g := M + \frac{I_g(1)}{\pi}$. Using Jensen's inequality, one gets
\begin{gather*}
    |\varphi(x_1 + iy_1, x')|g^{\pm1}(x)
    \le  C_g e^{\left(k_\varphi  + S_g(y_1)\right) y_1}\, \frac{y_1}{\pi} \int_{-\infty}^\infty \frac{|\varphi(t, x')|g^{\pm1}(t, x')}{(t - x_1)^2 + y_1^2}\, dt ,
\end{gather*}
where
\begin{equation}\label{Cg}
    C_g := e^{M + \frac{I_g(1)}{\pi}}.
\end{equation}
Estimate \eqref{est} now follows from Young's convolution inequality and \eqref{Lgp}.
\end{proof}
\begin{remark}\label{optimal}
Let $n = 1$, $g : \real \to [1, \infty)$ be a H\"older continuous submultiplicative function satisfying the Beurling--Domar condition,  $g(0) = 1$, and let
\begin{align*}
    w(x + iy)
    &:= \frac{y}{\pi} \int_{-\infty}^\infty \frac{\log{g(t)}}{(t - x)^2 + y^2}\, dt \\
    &\quad\mbox{} + \frac{i}{\pi} \int_{-\infty}^\infty \left(\frac{x - t}{(t - x)^2 + y^2} + \frac{t}{t^2 + 1}\right)\log{g(t)}\, dt ,
    \quad x \in \real, \; y > 0 .
\end{align*}
Then $\varphi(z) := e^{w(z)}$ is analytic in $z$ for $\Im z > 0$ and continuous up to $\real$,
\begin{gather*}
    |\varphi(x)| = e^{\Re(w(x))} = e^{\log{g(x)}} = g(x), \quad x \in \real,
\end{gather*}
see, e.g.\ \cite[Ch.~III, \S~1]{G06}, and
\begin{gather*}
    |\varphi(iy)|
    = e^{\Re(w(iy))} = \exp\left(\frac{y}{\pi} \int_{-\infty}^\infty \frac{\log{g(t)}}{t^2 + y^2}\,dt\right) = e^{S_g(y) y} , \quad   y \geq 1.
\end{gather*}
So,
\begin{gather*}
    k_\varphi = \limsup_{0 < y \to \infty} \frac{\log{|\varphi(iy)|}}{y} = \limsup_{y \to \infty} S_g(y) = 0
\end{gather*}
see \eqref{Jnfty}, and
\begin{align*}
    \|\varphi(\cdot + iy)\|_{L_{g^{-1}}^\infty(\real)}
    \ge \frac{|\varphi(iy)|}{g(0)}
      =  |\varphi(iy)| = e^{S_g(y) y}
    = e^{S_g(y) y} \|1\|_{L^\infty(\real)}
    &= e^{S_g(y) y} \|g^{-1}\varphi\|_{L^\infty(\real)} \\
    &= e^{S_g(y) y} \|\varphi\|_{L_{g^{-1}}^\infty(\real)} ,
\end{align*}
which shows that the factor $e^{S_g(y_1) y_1}$ in the right-hand side of \eqref{est} is optimal in this case.
\end{remark}

Clearly,
\begin{equation}\label{breve}
    S_{\breve{g}} = S_g, \quad C_{\breve{g}} = C_g,
\end{equation}
where $\breve{g}(x) := g(Ax)$ and $A \in O(n)$ is an arbitrary orthogonal matrix, see \eqref{Sg}, \eqref{Cg} and \eqref{M}.

\begin{theorem}\label{Tent}
    Let $g : \real^n\to [1, \infty)$ be a locally bounded, measurable submultiplicative function satisfying the Beurling--Domar condition \eqref{B-D}. Let $\varphi : \comp^n \to \comp$ be an entire function such that $\log{|\varphi(z)|} = O(|z|)$ for $|z|$ large, $z \in \comp^n$, and suppose that the restriction of $\varphi$ to $\real^n$ belongs to $L_{g^{\pm1}}^p(\real^n)$, $1 \le p \le \infty$. Then, for every multi-index $\alpha \in \mathbb{Z}_+^n$,
    \begin{equation}\label{entest}
        \left\|\left(\partial^\alpha\varphi\right)(\cdot + iy)\right\|_{L_{g^{\pm1}}^p(\real^n)}
        \le C_{\alpha} e^{\left(\kappa_\varphi(y/|y|) + S_g(|y|)\right) |y|} \|\varphi\|_{L_{g^{\pm1}}^p(\real^n)},
        \quad y \in \real^n ,
    \end{equation}
    where
    \begin{equation}\label{k2}
        \kappa_\varphi(\omega)
        := \sup_{x \in \real^n}\left(\limsup_{0 < t \to \infty} \frac{\log{|\varphi(x + i t\omega)|}}{t}\right) < \infty ,
        \quad \omega \in \Ss^{n - 1} ,
    \end{equation}
    and the constant $C_{\alpha}  \in (0, \infty)$ depends only on $\alpha$ and $g$.
\end{theorem}
\begin{proof}(Cf.\ the proof of Lemma 9.29 in \cite{LG86}.)
Take any $y \in \real^n\setminus\{0\}$. There exist an orthogonal matrix $A \in O(n)$ such that $A\mathbf{e}_1 = \omega := y/|y|$, see \eqref{ej}. Let $\breve{\varphi}(z) := \varphi(Az)$, $z \in \comp^n$, and $\breve{g}(x) := g(Ax)$, $x \in \real^n$. Then $\breve{\varphi} : \comp^n \to \comp$ is an entire function, and one can apply to it Lemma \ref{x1} with $\breve{g}$ in place of $g$, see \eqref{breve}.

For any $x \in \real^n$, one has $\varphi(x + iy) = \breve{\varphi}\left(\tilde{x} + i|y|\mathbf{e}_1\right) = \breve{\varphi}\left(\tilde{x}_1 + i|y|, \tilde{x}_2, \dots, \tilde{x}_n\right)$, where $\tilde{x} := A^{-1}x$. Hence
\begin{align*}
    \|\varphi(\cdot + iy)\|_{L_{g^{\pm1}}^p(\real^n)}
    &= \left\|\breve{\varphi}(\cdot + i|y|, \cdot)\right\|_{L_{\breve{g}^{\pm1}}^p(\real^n)}\\
    &\le  C_{\breve{g}} e^{\left(k_{\breve{\varphi}}  + S_{\breve{g}}(|y|)\right) |y|} \left\|\breve{\varphi}\right\|_{L_{\breve{g}^{\pm1}}^p(\real^n)} \\
    &\le  C_g e^{\left(\kappa_\varphi(y/|y|) + S_g(|y|)\right) |y|}  \left\|\breve{\varphi}\right\|_{L_{\breve{g}^{\pm1}}^p(\real^n)}\\
    &= C_g e^{\left(\kappa_\varphi(y/|y|) + S_g(|y|)\right) |y|} \|\varphi\|_{L_{g^{\pm1}}^p(\real^n)},
\end{align*}
see \eqref{breve}, which proves \eqref{entest} for $\alpha = 0$ and $y \neq  0$. This estimate is trivial for $\alpha = 0$ and $y = 0$.

Iterating the standard Cauchy integral formula for one complex variable, one gets
\begin{gather*}
    \varphi(\zeta)
    = \frac{1}{(2\pi)^n} \int_0^{2\pi}\dots \int_0^{2\pi} \frac{\varphi(z_1 + e^{i\theta_1}, \dots, z_n + e^{i\theta_n})}{\prod_{k = 1}^n (z_k + e^{i\theta_k} - \zeta_k)}\,
    \left(\prod_{k = 1}^n e^{i\theta_k}\right) d\theta_1\dots d\theta_n ,
    \\
    \zeta \in \Delta(z) := \left\{\eta \in \comp^n : \ |\eta_k - z_k| < 1, \ k = 1, \dots, n\right\} , \ z \in \comp^n,
\intertext{cf.~\cite[Ch.~1, \S~1]{LG86}), which implies}
    \partial^\alpha\varphi(\zeta)
    = \frac{\alpha!}{(2\pi)^n} \int_0^{2\pi}\dots \int_0^{2\pi} \frac{\varphi(z_1 + e^{i\theta_1}, \dots, z_n + e^{i\theta_n})}{\prod_{k = 1}^n (z_k + e^{i\theta_k} - \zeta_k)^{\alpha_k + 1}}\, \left(\prod_{k = 1}^n e^{i\theta_k}\right) d\theta_1\dots d\theta_n .
\end{gather*}
Hence,
\begin{gather}
    \partial^\alpha\varphi(z)
    = \frac{\alpha!}{(2\pi)^n} \int_0^{2\pi}\dots \int_0^{2\pi} \frac{\varphi(z_1 + e^{i\theta_1}, \dots, z_n + e^{i\theta_n})}{\prod_{k = 1}^n e^{i\alpha_k\theta_k}}\, d\theta_1\dots d\theta_n , \notag
\intertext{and}\label{al}
    \left|\partial^\alpha\varphi(z)\right|
    \le \frac{\alpha!}{(2\pi)^n} \int_0^{2\pi}\dots \int_0^{2\pi} \left|\varphi(z_1 + e^{i\theta_1}, \dots, z_n + e^{i\theta_n})\right|\, d\theta_1\dots d\theta_n .
\end{gather}

Since $g \ge 1$ is locally bounded,
\begin{gather*}
    1 \le M_1 := \sup_{|s_k| \le 1,\,  k = 1, \dots, n} g(s) < \infty .
\end{gather*}
Then it follows from \eqref{submult} that
\begin{equation}\label{M1est}
    g^{\pm1}(x_1 - \cos{\theta_1}, \dots, x_n - \cos{\theta_n}) \le M_1 g^{\pm1}(x) .
\end{equation}
According to the conditions of the theorem, there exists a constant $c_\varphi \in (0, \infty)$ such that
$\log{|\varphi(\zeta)|} \le c_\varphi |\zeta|$ for $|\zeta|$ large.
Then $\kappa_\varphi(\omega) \le c_\varphi$, see \eqref{k2}. Let
$\varphi_y := \varphi(\cdot + iy)$, $y = (\Im z_1, \dots, \Im z_n)$.
Then, similarly to the above inequality, $\kappa_{\varphi_y}(\omega) \le c_\varphi$. Applying \eqref{entest} with $\alpha = 0$ to the function $\varphi_y$ in place of $\varphi$ and using \eqref{Sgdecr}, \eqref{M1est}, one derives from \eqref{al}
\begin{align*}
    &\left\|\left(\partial^\alpha\varphi\right)(\cdot + iy)\right\|_{L_{g^{\pm1}}^p(\real^n)} \\
    &\le \frac{\alpha!}{(2\pi)^n} \int_0^{2\pi}\dots \int_0^{2\pi}
        \left\|\varphi(\cdot + iy_1 + e^{i\theta_1}, \dots, \cdot + iy_n + e^{i\theta_n})\right\|_{L_{g^{\pm1}}^p(\real^n)}
        \, d\theta_1\dots d\theta_n  \\
    &\le \frac{\alpha!}{(2\pi)^n} \int_0^{2\pi}\dots \int_0^{2\pi} M_1
        \left\|\varphi(\cdot + iy_1 + i\sin{\theta_1}, \dots, \cdot + iy_n + i\sin{\theta_n})\right\|_{L_{g^{\pm1}}^p(\real^n)}
        \, d\theta_1\dots d\theta_n   \\
    &\le \frac{\alpha!}{(2\pi)^n} \int_0^{2\pi}\dots \int_0^{2\pi} M_1 C_0 e^{\left(c_\varphi + S_g(1)\right) \sqrt{n}}
        \|\varphi(\cdot + iy)\|_{L_{g^{\pm1}}^p(\real^n)}
        \, d\theta_1\dots d\theta_n  \\
    &= \alpha! M_1 C_0 e^{\left(c_\varphi + S_g(1)\right) \sqrt{n}} \|\varphi(\cdot + iy)\|_{L_{g^{\pm1}}^p(\real^n)} .
\end{align*}
Applying \eqref{entest} with $\alpha = 0$ again, one gets
\begin{gather*}
    \left\|\left(\partial^\alpha\varphi\right)(\cdot + iy)\right\|_{L_{g^{\pm1}}^p(\real^n)}
    \le \alpha! M_1 C_0^2 e^{\left(c_\varphi + S_g(1)\right) \sqrt{n}} e^{\left(\kappa_\varphi(y/|y|) + S_g(|y|)\right) |y|}  \|\varphi\|_{L_{g^{\pm1}}^p(\real^n)} .
\qedhere
\end{gather*}
\end{proof}

\begin{corollary}\label{ent}
    Let $g : \real^n\to [1, \infty)$ be a locally bounded, measurable submultiplicative function satisfying the Beurling--Domar condition \eqref{B-D}. Let $\varphi : \comp^n \to \comp$ be an entire function such that $\log{|\varphi(z)|} = O(|z|)$ for $|z|$ large, $z \in \comp^n$, and that the restriction of $\varphi$ to $\real^n$ belongs to $L_{g^{\pm1}}^p(\real^n)$, $1 \le p \le \infty$. Then for every multi-index $\alpha \in \mathbb{Z}_+^n$ and every $\varepsilon > 0$,
\begin{equation}\label{entesteps}
    \left\|\left(\partial^\alpha\varphi\right)(\cdot + iy)\right\|_{L_{g^{\pm1}}^p(\real^n)}
    \le C_{\alpha, \varepsilon} e^{(\kappa_\varphi(y/|y|) + \varepsilon)|y|} \|\varphi\|_{L_{g^{\pm1}}^p(\real^n)}, \quad  y \in \real^n ,
\end{equation}
where $\kappa_\varphi$ is defined by \eqref{k2}, and the constant $C_{\alpha, \varepsilon}  \in (0, \infty)$ depends only on $\alpha$, $\varepsilon$, and $g$.
\end{corollary}
\begin{proof}
It follows from \eqref{Jnfty} that for every $\varepsilon > 0$, there exists some $c_\varepsilon$ such that
\begin{gather*}
    S_g(|y|) |y| \le c_\varepsilon + \varepsilon |y| \quad\text{for all}\quad y \in \real^n .
\end{gather*}
Hence, \eqref{entest} implies \eqref{entesteps}.
\end{proof}

\section{Main results}\label{Ch4}

We will use the notation $\widetilde{g}(x) := g(-x)$, $x \in \real^n$.  It follows from submultiplicativity of $\widetilde{g}$ that
$L^1_{\widetilde{g}}(\mathbb{R}^n)$ is a convolution algebra.

Taking $y - x$ in place of $y$ in \eqref{submult} and rearranging, one gets
\begin{equation}\label{ggg}
    \frac{1}{g(x)} \le \frac{g(y - x)}{g(y)}.
\end{equation}
Using this inequality, one can easily show that $f\ast u \in L^\infty_{g^{-1}}(\real^n)$ for every $f \in L^\infty_{g^{-1}}(\real^n)$ and $u \in L^1_{\widetilde{g}}(\real^n)$. The Fubini-Tonelli theorem implies that
\begin{equation}\label{assoc}
    f\ast (v\ast u) = (f\ast v)\ast u
    \quad\text{for all}\quad f \in L^\infty_{g^{-1}}(\real^n)
    \quad\text{and}\quad v, u \in L^1_{\widetilde{g}}(\real^n).
\end{equation}
Let $A_{\widetilde{g}} := \left\{c\delta +  g \mid c\in\mathbb{C},\, g\in L^1_{\widetilde{g}}(\mathbb{R}^n)\right\}$, where $\delta$ is the Dirac measure
on $\mathbb{R}^n$. This is the algebra
$L^1_{\widetilde{g}}(\mathbb{R}^n)$ with a unit attached, cf.\ Rudin \cite[10.3(d), 11.13(e)]{Ru73}. Clearly, \eqref{assoc} holds for any $v, u \in A_{\widetilde{g}}$.

\begin{theorem}\label{Taub}
    Let $g : \real^n\to [1, \infty)$ be a locally bounded, measurable submultiplicative function satisfying the Beurling--Domar condition \eqref{B-D}, $f \in L^\infty_{g^{-1}}(\real^n)$, and $Y$  be a linear subspace of $L^1_{\widetilde{g}}(\real^n)$ such that
\begin{gather}\label{eq:2.1}
    f \ast v = 0 \quad\text{for every\ } v \in Y .
\end{gather}
Suppose the set
\begin{gather}\label{eq:2.2}
    Z(Y) := \bigcap_{v \in Y} \left\{\xi \in \real^n \mid \widehat{v}(\xi) = 0\right\}
\end{gather}
is bounded, and $u \in L^1_{\widetilde{g}}(\real^n)$ is such that $\widehat{u} = 1$ in a neighbourhood of $Z(Y)$. Then $f = f\ast u$.  If $Z(Y) = \emptyset$, then $f = 0$.
\end{theorem}
\begin{proof}
 In order to  prove the equality $f = f\ast u$, it is sufficient to show that
\begin{equation}\label{h}
    \langle f, h\rangle = \langle f\ast u, h\rangle \quad\text{for every\ } h \in L^1_g(\real^n) .
\end{equation}
Since the set of functions $h$ with compactly supported Fourier transforms $\widehat{h}$ is dense in $L^1_g(\real^n)$, see \cite[Thm.~1.52 and~2.11]{D56}, it is enough to prove \eqref{h} for such $h$. Further,
\begin{gather*}
    \langle f, h\rangle = \big(f\ast \widetilde{h}\big)(0) .
\end{gather*}
So, we have to show only that
\begin{equation}\label{w}
    f\ast w = f\ast u\ast w
\end{equation}
for every $w \in L^1_{\widetilde{g}}(\real^n)$ with compactly supported Fourier transform $\widehat{w}$. Take any such $w$ and choose $R > 0$ such that the support of $\widehat{w}$ lies in $B_R := \left\{\xi \in \real^n : \ |\xi| \le R\right\}$. It is clear that $\widetilde{g}$ satisfies the Beurling--Domar condition. Then there exists $u_R \in L^1_{\widetilde{g}}(\real^n)$ such that $0 \le \widehat{u_R} \le 1$, $\widehat{u_R}(\xi) = 1$ for $|\xi| \le R$, and $\widehat{u_R}(\xi) = 0$ for $|\xi| \ge R + 1$, see \cite[Lemma 1.24]{D56}.

If $Z(Y) \neq \emptyset$,  let $V$ be an open neighbourhood of $Z(Y)$ such that $\widehat{u} = 1$ in $V$. Similarly to the above, there exists $u_0 \in L^1_{\widetilde{g}}(\mathbb{R}^n)$ such that $0 \le \widehat{u_0} \le 1$, $\widehat{u_0} = 1$ in a neighbourhood $V_0 \subset V$ of $Z(Y)$, and $\widehat{u_0} = 0$ outside $V$, see \cite[Lemma 1.24]{D56}.  If $Z(Y) = \emptyset$, one can take $u = u_0 = 0$
and $V_0 = \emptyset$ below.

Since $Y$ is a linear subspace, for every $\eta \in B_{R + 1}\setminus V_0 \subset \real^n\setminus Z(Y)$, there exists $v_\eta \in Y$ such that $\widehat{v_\eta}(\eta) = 1$. Since $v_\eta \in L^1(\real^n)$, $\widehat{v_\eta}$ is continuous, and there is a neighbourhood $V_\eta$ of $\eta$ such that $\left|\widehat{v_\eta}(\xi) - 1\right| < 1/2$ for all $\xi \in V_\eta$. Similarly to the above, there exists $u_\eta \in L^1_{\widetilde{g}}(\real^n)$ such that $\Re\left(\widehat{v_\eta}\widehat{u_\eta}\right) \ge 0$, and $\Re\left(\widehat{v_\eta}\widehat{u_\eta}\right) > \frac12$ in a neighbourhood $V^0_\eta \subset V_\eta$ of $\eta$.

Since $B_{R + 1}\setminus V_0$ is compact, the open cover $\{V^0_\eta\}_{\eta \in B_{R + 1}\setminus V_0}$ has a finite subcover. So, there exist functions $v_j \in Y$ and $u_j \in L^1_{\widetilde{g}}(\real^n)$, $j = 1, \dots, N$ such that
\begin{gather*}
    \Re\left(\sigma\right) > \frac12,
    \quad\text{where}\quad
    \sigma := \widehat{u_0} + \sum_{j = 1}^N \widehat{v_j} \widehat{u_j} + 1 - \widehat{u_R} .
\end{gather*}
Then there exists $\upsilon \in  A_{\widetilde{g}} $ such that $\widehat{\upsilon} = 1/\sigma$, see \cite[Thm.~1.53]{D56}.

Since $\widehat{u_0}(1 - \widehat{u}) = 0$ and $\left(1 - \widehat{u_R}\right)\widehat{w} = 0$, one has
\begin{align*}
    \left(\widehat{u} + \sum_{j = 1}^N \widehat{v_j} \widehat{u_j}\widehat{\upsilon}\left(1 - \widehat{u}\right)\right)\widehat{w}
    &= \big(\widehat{u} + \left(\sigma - (\widehat{u_0} + 1 - \widehat{u_R})\right)\widehat{\upsilon}\left(1 - \widehat{u}\right)\big)\widehat{w} \\
    &=  \big(\widehat{u} + \left(1 - \widehat{u}\right) - (\widehat{u_0} + 1 - \widehat{u_R})\widehat{\upsilon}\left(1 - \widehat{u}\right)\big)\widehat{w}\\
    &= \big(1 - (1 - \widehat{u_R})\widehat{\upsilon}\left(1 - \widehat{u}\right)\big)\widehat{w} \\
    &= \widehat{w} -  (1 - \widehat{u_R})\widehat{w}\widehat{\upsilon}\left(1 - \widehat{u}\right)
    = \widehat{w} .
\end{align*}
It now follows from \eqref{assoc} and \eqref{eq:2.1} that
\begin{align*}
    f\ast w
    &= f\ast \left(u + \sum_{j = 1}^N v_j\ast u_j\ast \left(\upsilon - \upsilon\ast u\right)\right)\ast w \\
    &= f\ast u\ast w + f\ast \left(\sum_{j = 1}^N v_j\ast u_j\ast \left(\upsilon - \upsilon\ast u\right)\right)\ast w \\
    &= f\ast u\ast w +  \sum_{j = 1}^N (f\ast v_j)\ast u_j\ast \left(\upsilon - \upsilon\ast u\right)\ast w = f\ast u\ast w.
\end{align*}
If $Z(Y) = \emptyset$, one can take $ u = 0$, and the equality $f = f\ast u$ means that $f = 0$.
\end{proof}


For a bounded set $E \subset \real^n$, let $\mathrm{conv}(E)$ denote its closed convex hull, and $H_E$ denote its support function:
\begin{gather*}
    H_E(y) :=  \sup_{\xi \in E}\, y\cdot\xi =  \sup_{\xi \in \mathrm{conv}(E)} y\cdot\xi, \quad y \in \real^n .
\end{gather*}
Clearly, $H_E$ is positively homogeneous and convex:  for   all $x,y\in\real^n$ and $\tau\geq 0$ we have
\begin{gather*}
    H_E(\tau y) = \tau H_E(y), \quad
    H_E(y + x)  \le H_E(y) + H_E(x) .
\end{gather*}
For every positively homogeneous convex function $H$,
\begin{equation}\label{CompConv}
    K := \left\{\xi \in \real^n \mid  y\cdot\xi \le H(y) \text{\ \ for all\ \ } y \in \real^n\right\}
\end{equation}
is the unique convex compact set such that $H_K = H$, see, e.g.\ \cite[Thm.~4.3.2]{H83}.
\begin{theorem}\label{Ent}
    Let $g$, $f$, and $Y$ satisfy the conditions of Theorem \ref{Taub}, and let
\begin{equation}\label{SuppF}
    \mathcal{H}_Y(y) := H_{Z(Y)}(-y) =  \sup_{\xi \in Z(Y)} (-y)\cdot\xi = -\inf_{\xi \in Z(Y)} y\cdot\xi  , \quad y \in \real^n .
\end{equation}
    Then $f$ admits analytic continuation to an entire function $f : \comp^n \to \comp$ such that for every multi-index $\alpha \in \mathbb{Z}_+^n$,
\begin{equation}\label{fentest}
    \left\|\left(\partial^\alpha f\right)(\cdot + iy)\right\|_{L_{g^{-1}}^\infty(\real^n)}
    \le C_{\alpha} e^{\mathcal{H}_Y(y) + S_g(|y|)|y|} \|f\|_{L_{g^{-1}}^\infty(\real^n)},
    \quad y \in \real^n,
\end{equation}
see \eqref{Sg}, \eqref{Jnfty}, where the constant $C_{\alpha}  \in (0, \infty)$ depends only on $\alpha$ and $g$.
\end{theorem}

\begin{proof}
Take any $\varepsilon > 0$. There exists $u \in L^1_{\widetilde{g}}(\real^n)$ such that $\widehat{u} = 1$ in a neighbourhood of $Z(Y)$, and $\widehat{u} = 0$ outside the $\frac{\varepsilon}{2}$-neighbourhood of $Z(Y)$, see \cite[Lemma 1.24]{D56}. It follows from the Paley--Wiener--Schwartz theorem, see, e.g.\ \cite[Thm.~7.3.1]{H83} that $u = \mathcal{F}^{-1} \widehat{u}$ admits analytic continuation to an entire function $u : \comp^n \to \comp$ satisfying the estimate
\begin{gather*}
    |u(x + iy)| \le c_\varepsilon e^{\mathcal{H}_Y(y) + \varepsilon|y|/2} \quad\text{for all}\quad x, y \in \real^n
\end{gather*}
with some constant $c_\varepsilon \in (0, \infty)$. So, $u$ satisfies the conditions of Corollary \ref{ent} with $\widetilde{g}$ in place of $g$, and
\begin{equation}\label{uentest}
    \|u(\cdot + iy)\|_{L_{\widetilde{g}}^1(\real^n)}
    \le C_{0, \varepsilon/2}\, e^{\mathcal{H}_Y(y) + \varepsilon|y|} \|u\|_{L_{\widetilde{g}}^1(\real^n)},
    \quad  y \in \real^n .
\end{equation}
Since
\begin{gather*}
    f(x) = \int_{\real^n} u(x - s) f(s)\, ds,
\end{gather*}
see Theorem \ref{Taub}, $f$ admits analytic continuation
\begin{gather*}
    f(x + iy) := \int_{\real^n} u(x + iy - s) f(s)\, ds,
\end{gather*}
see Corollary \ref{ent}, and
\begin{align*}
    \|f(\cdot + iy)\|_{L_{g^{-1}}^\infty(\real^n)}
    &\le \|u(\cdot + iy)\|_{L_{\widetilde{g}}^1(\real^n)} \|f\|_{L_{g^{-1}}^\infty(\real^n)} \\
    &\le C_{0, \varepsilon/2}\, e^{\mathcal{H}_Y(y) + \varepsilon|y|} \|u\|_{L_{\widetilde{g}}^1(\real^n)} \|f\|_{L_{g^{-1}}^\infty(\real^n)}\\
    &=: M_\varepsilon e^{\mathcal{H}_Y(y) + \varepsilon|y|} \|f\|_{L_{g^{-1}}^\infty(\real^n)},
\end{align*}
see \eqref{ggg}. Since
\begin{gather*}
    \frac{|f(x + iy)|}{g(x)}
    \le M_\varepsilon e^{\mathcal{H}_Y(y) + \varepsilon|y|} \|f\|_{L_{g^{-1}}^\infty(\real^n)} ,
\end{gather*}
one has $\log{|f(x + iy)|} = O(|x + iy|)$ for $|x + iy|$ large, see \eqref{expbdd}, and
\begin{align*}
    \limsup_{0 < t \to \infty} \frac{\log{|f(x + i t\omega)|}}{t}
    &\le \limsup_{0 < t \to \infty} \frac{\log\left(M_\varepsilon g(x) \|f\|_{L_{g^{-1}}^\infty(\real^n)}\right) + t\mathcal{H}_Y(\omega) + \varepsilon t}{t} \\
    &= \mathcal{H}_Y(\omega) + \varepsilon .
\end{align*}
Hence,
\begin{gather*}
    \kappa_f(\omega) := \sup_{x \in \real^n}\left(\limsup_{0 < t \to \infty} \frac{\log{|f(x + i t\omega)|}}{t}\right) \le \mathcal{H}_Y(\omega) + \varepsilon
\end{gather*}
for every $\varepsilon > 0$, i.e.
\begin{gather*}
    \kappa_f(\omega) \le  \mathcal{H}_Y(\omega) .
\end{gather*}
So, \eqref{fentest} follows from Theorem \ref{Tent}.
\end{proof}

\begin{theorem}\label{compZ}
    Let $g : \real^n\to [1, \infty)$ be a locally bounded, measurable submultiplicative function satisfying the Beurling--Domar condition \eqref{B-D}, and  let $m\in C(\real^n)$ be such that the Fourier multiplier operator
    \begin{gather*}
        C_c^\infty(\real^n)\ni \phi \mapsto \widetilde m(D)\phi :=\mathcal{F}^{-1}(\widetilde m\widehat\phi)
    \end{gather*}
    maps $C_c^\infty(\real^n)$ into $L^1_g(\real^n)$. Suppose $f \in L^\infty_{g^{-1}}(\real^n)$ is such that $m(D)f=0$ as a distribution, i.e.\
    \begin{gather}\label{eq:3.3}
        \langle f, \widetilde m(D)\phi\rangle = 0 \quad\text{for all\ } \phi \in C_c^\infty(\real^n).
    \end{gather}
    If $K := \left\{\eta \in \real^n \mid  m(\eta) = 0\right\}$ is compact, then $f$ admits analytic continuation to an entire function $f : \comp^n \to \comp$ such that for every multi-index $\alpha \in \mathbb{Z}_+^n$,
    \begin{equation}\label{Hentest}
        \left\|\left(\partial^\alpha f\right)(\cdot + iy)\right\|_{L_{g^{-1}}^\infty(\real^n)}
        \le C_{\alpha} e^{H(y) + S_g(|y|)|y|} \|f\|_{L_{g^{-1}}^\infty(\real^n)},
        \quad y \in \real^n,
    \end{equation}
    see \eqref{Sg}, \eqref{Jnfty}, where $H(y) := H_K(-y)$, and the constant $C_{\alpha}  \in (0, \infty)$ depends only on $\alpha$ and $g$.

    Conversely, if every $f \in L^\infty(\real^n)$ satisfying \eqref{eq:3.3} admits analytic continuation to an entire function $f : \comp^n \to \comp$ such that
    \begin{equation}\label{fentest1}
        \|f(\cdot + iy)\|_{L_{g^{-1}}^\infty(\real^n)}
        \le M_\varepsilon e^{H(y) + \varepsilon|y|} \|f\|_{L_{g^{-1}}^\infty(\real^n)},
        \quad  y \in \real^n ,
    \end{equation}
    holds for every $\varepsilon > 0$ with a constant $M_\varepsilon \in (0, \infty)$ that depends only on $\varepsilon$, $m$, and $g$, then $\left\{\eta \in \real^n \mid  m(\eta) = 0\right\} \subseteq K$, where $K$ is the unique convex compact set such that $H_K(y) = H(-y)$; cf.\ \eqref{CompConv}.
\end{theorem}
\begin{proof}
Denote by $(T_\upsilon \phi)(x) := \phi(x - \upsilon)$, $x, \upsilon \in \real^n$ the shift by $\upsilon$. Since $T_\upsilon \phi \in C_c^\infty(\real^n)$ for every $\phi\in C_c^\infty(\real^n)$ and all $\upsilon \in \real^n$, it follows from \eqref{eq:3.3} that
\begin{gather*}
    \big(f \ast \widetilde{\widetilde{m}(D) \phi}\big)(\upsilon)
    = \langle f, T_\upsilon\widetilde{m}(D) \phi\rangle
    = \langle f, \widetilde{m}(D) \left(T_\upsilon\phi\right)\rangle
    = 0 \quad\text{for all\ } \upsilon \in \real^n .
\end{gather*}
Hence,
\begin{gather*}
    f \ast \widetilde{\widetilde{m}(D) \phi} = 0 \quad\text{for all\ } \phi \in C_c^\infty(\real^n) .
\end{gather*}
It is easy to see that
\begin{align*}
    \bigcap_{\phi \in C_c^\infty(\real^n)} \bigg\{\eta \in \real^n \mid \widehat{\widetilde{\widetilde{m}(D) \phi}}(\eta) = 0\bigg\}
    &= \bigcap_{\phi \in C_c^\infty(\real^n)} \bigg\{\eta \in \real^n \mid  \widehat{\widetilde{m}(D) \phi}(-\eta) = 0\bigg\} \\
    &= \bigcap_{\phi \in C_c^\infty(\real^n)} \left\{\eta \in \real^n \mid  m(\eta) \widehat{\phi}(-\eta) = 0\right\}\\
    &= \left\{\eta \in \real^n \mid m(\eta) = 0\right\} = K .
\end{align*}
Applying Theorem \ref{Ent} with
\begin{gather*}
    Y := \left\{\widetilde{\widetilde m(D) \phi}\ \big| \ \phi \in C_c^\infty(\real^n)\right\} \subset L_{\widetilde g}^1(\real^n)
\end{gather*}
and $Z(Y) = K$, one gets \eqref{Hentest}.

For the converse direction, we assume the contrary, i.e.\ that the zero-set $\{\eta\in\real^n \mid m(\eta)=0\}$ contains some $\gamma\not\in K$, see \eqref{CompConv}. Then there exists a $y_0 \in \real^n\setminus\{0\}$ such that $y_0\cdot \gamma > H_K(y_0) = H(-y_0)$. It is easy to see that $f(x) := e^{ix\cdot\gamma}$ satisfies $m(D)e^{ix\cdot\gamma} = e^{ix\cdot\gamma} m(\gamma) = 0$ for all $x\in\real^n$. Take $\varepsilon < (y_0\cdot \gamma - H(-y_0))/|y_0|$. Clearly, $f \in L^\infty(\real^n)$, and
\begin{align*}
    \frac{\|f(\cdot - i\tau y_0)\|_{L_{g^{-1}}^\infty(\real^n)}}{e^{H(-\tau y_0) + \varepsilon|\tau y_0|} \|f\|_{L_{g^{-1}}^\infty(\real^n)}}
    = \frac{e^{\tau(y_0\cdot \gamma)}}{e^{\tau(H(-y_0) + \varepsilon|y_0|)}}
    = e^{\tau(y_0\cdot \gamma - H(-y_0) - \varepsilon|y_0|)}  \xrightarrow[\tau\to\infty]{} \infty.
\end{align*}
So, $f$ does not satisfy \eqref{fentest1}.
\end{proof}

\begin{corollary}\label{th:3.1}
    Let $g : \real^n\to [1, \infty)$ be a locally bounded, measurable submultiplicative function satisfying the Beurling--Domar condition \eqref{B-D} and  let $m\in C(\real^n)$ be such that the Fourier multiplier operator
    \begin{gather*}
        C_c^\infty(\real^n)\ni \phi \mapsto \widetilde m(D)\phi :=\mathcal{F}^{-1}(\widetilde m\widehat\phi)
    \end{gather*}
    maps $C_c^\infty(\real^n)$ into $L^1_g(\real^n)$. Suppose $f \in L^\infty_{g^{-1}}(\real^n)$ is such that $m(D)f=0$ as a distribution, i.e.\ \eqref{eq:3.3} holds. If $\left\{\eta \in \real^n \mid  m(\eta) = 0\right\} = \{0\}$, then $f$ admits analytic continuation to an entire function $f : \comp^n \to \comp$ such that for every multi-index $\alpha \in \mathbb{Z}_+^n$,
\begin{equation}\label{Hentest0}
    \left\|\left(\partial^\alpha f\right)(\cdot + iy)\right\|_{L_{g^{-1}}^\infty(\real^n)}
    \le C_{\alpha} e^{S_g(|y|)|y|} \|f\|_{L_{g^{-1}}^\infty(\real^n)},
    \quad y \in \real^n ,
\end{equation}
    where the constant $C_{\alpha}  \in (0, \infty)$ depends only on $\alpha$ and $g$. If $\left\{\eta \in \real^n \mid  m(\eta) = 0\right\} = \emptyset$, then $f = 0$.

    Conversely, if every $f \in L^\infty(\real^n)$ satisfying \eqref{eq:3.3} admits analytic continuation to an entire function $f : \comp^n \to \comp$ such that
\begin{equation}\label{fentest0}
    \|f(\cdot + iy)\|_{L_{g^{-1}}^\infty(\real^n)}
    \le M_\varepsilon e^{\varepsilon|y|} \|f\|_{L_{g^{-1}}^\infty(\real^n)},
    \quad  y \in \real^n ,
\end{equation}
    holds for every $\varepsilon > 0$ with a constant $M_\varepsilon \in (0, \infty)$ that depends only on $\varepsilon$, $m$, and $g$, then $\left\{\eta \in \real^n \mid  m(\eta) = 0\right\} \subseteq \{0\}$.
\end{corollary}
\begin{proof}
    The only part that does not follow immediately from Theorem \ref{compZ} is that $f = 0$ in the case $\left\{\eta \in \mathbb{R}^n \mid  m(\eta) = 0\right\} = \emptyset$. In this case, one can take the same $Y$ as in the proof of Theorem \ref{compZ}, note that $Z(Y) = \emptyset$ and apply  Theorem \ref{Taub}  to conclude that $f = 0$. (It is instructive to compare this result to \cite[Proposition 2.2]{Kan00}.)
\end{proof}


\begin{remark}
    The condition that $\widetilde{m}(D)$ maps $C^\infty_c(\real^n)$ to $L^1_g(\real^n)$ is satisfied if $m$ is a linear combination of terms of the form $ab$, where $a = F\mu$,\ $\mu$ is a finite complex Borel measure on $\real^n$ such that
    \begin{gather*}
        \int_{\real^n} \widetilde{g}(y)\, |\mu|(dy)<\infty,
    \end{gather*}
    and $b$ is the Fourier transform of a compactly supported distribution. Indeed, it is easy to see that $\widetilde{b}(D)$ maps $C^\infty_c(\real^n)$ into itself, while the convolution operator $\phi \mapsto \widetilde{\mu}\ast\phi$ maps $C^\infty_c(\real^n)$ to $L^1_g(\real^n)$.

    A particular example is the characteristic exponent of a L\'evy process  (this is a stochastic process with stationary and independent increments, such that the trajectories are right-continuous with finite left limits, see e.g.\ Sato~\cite{sato})
    \begin{align*}
    m(\xi)
    &= -ib\cdot\xi + \frac 12 \xi\cdot Q\xi + \int_{0<|y|<1} \left(1-e^{iy\cdot\xi} + iy\cdot\xi\right) \nu(dy)  \\
    &\quad\mbox{} + \int_{|y|\geq 1} \left(1-e^{iy\cdot\xi}\right) \nu(dy) ,
    \end{align*}
    where $b\in\mathbb{R}^n$, $Q\in\mathbb{R}^{n\times n}$ is a symmetric positive semidefinite matrix, and $\nu$ is a measure on $\mathbb{R}^n\setminus\{0\}$ such that $\int_{0<|y|<1} |y|^{2}\,\nu(dy) + \int_{|y|\geq 1}g(y)\,\nu(dy)<\infty$. More generally, one can take
    \begin{align*}
        m(\xi)
        &= \sum_{|\alpha|=0}^{2s} c_\alpha \frac{i^{|\alpha|}}{\alpha!} \xi^\alpha
        + \int_{0<|y|<1}\bigg[1-e^{iy\cdot\xi}+\sum_{|\alpha|=0}^{2s-1}\frac{i^{|\alpha|}}{\alpha!} y^\alpha\xi^\alpha\bigg]\,\nu(dy) \\
        &\quad\mbox{} + \int_{|y|\geq 1} \left(1-e^{iy\cdot\xi}\right)\nu(dy)
    \end{align*}
    with $s\in\mathbb{N}$, $c_\alpha\in\mathbb{R}$, and a measure $\nu$ on $\mathbb{R}^n\setminus\{0\}$ such that $\int_{0<|y|<1} |y|^{2s}\,\nu(dy) + \int_{|y|\geq 1}g(y)\,\nu(dy)<\infty$. (As usual, for any $\alpha\in\mathbb{N}_0^n$ and $\xi \in\mathbb{R}^n$, we define $\alpha!:= \prod_1^n \alpha_k!$ and $\xi^\alpha := \prod_1^n \xi_k^{\alpha_k}$.) Functions of this type appear naturally in positivity questions related to generalised functions (see, e.g.~\cite[Ch.~II, \S 4]{gel-vil64} or \cite[Ch.~8]{wendland}). Some authors call the function $-m$ for such an $m$ (under suitable additional conditions on the $c_\alpha$'s) a \emph{conditionally positive definite function}.
\end{remark}

\begin{remark}\label{cases}
We are mostly interested in super-polynomially growing weights as polynomially growing ones have been dealt with in our previous paper \cite{BSS23}. Nevertheless, it is instructive to look at the behaviour of the factor $e^{S_g(|y|)|y|}$ for  typical  super-polynomially, polynomially, and sub-polynomially growing weights.

It follows from \eqref{gamma} that if $g(x) = e^{|x|/\log^\gamma(e + |x|)}$, $\gamma > 1$, then there exists a constant $C_\gamma$ such that
\begin{align*}
    e^{S_g(|y|)|y|}
    &\le C_\gamma e^{\frac{1}{\pi} |y| \log^{-\gamma}(e + |y|)\left(1 + \frac{2}{\gamma - 1} \log(e + |y|)\right)} \\
    &= C_\gamma \left(e^{|y|/\log^\gamma(e + |y|)}\right)^{\frac{1}{\pi} \left(1 + \frac{2}{\gamma - 1} \log(e + |y|)\right)}\\
    &= C_\gamma (g(y))^{\frac{1}{\pi} \left(1 + \frac{2}{\gamma - 1} \log(e + |y|)\right)}  .
\end{align*}
Similarly, if $g(x)  = e^{a |x|^b}$, $a \ge 0$, $b \in [0, 1)$, then \eqref{ab} implies
\begin{equation}\label{gab}
    e^{S_g(|y|)|y|} = e^{a |y|^{b}\left(\sin\left(\frac{1 - b}{2} \pi\right)\right)^{-1}}
    = (g(y))^{\left(\sin\left(\frac{1 - b}{2} \pi\right)\right)^{-1}} .
\end{equation}
If $g(x)  = (1 + |x|)^s$, $s \ge 0$, then \eqref{s} implies
\begin{equation}\label{gpol}
    e^{S_g(|y|)|y|} \le e^{c_1 s + s \log(1 + |y|)}
    = C_s (1 + |y|)^{s} = C_s\, g(y) .
\end{equation}
Finally, if $g(x)  = (\log(e + |x|))^t$, $t \ge 0$, then \eqref{t} implies
\begin{gather*}
    e^{S_g(|y|)|y|} \le  e^{c_2 t + t \log\log(e + |y|)}
    = C_t (\log(e + |y|))^{t} = C_t\, g(y)  .
\end{gather*}
\end{remark}

\begin{remark}
If $g$ is polynomially bounded in Corollary \ref{th:3.1}, then it follows from \eqref{Hentest0} and \eqref{gpol} that $f$ is a polynomially bounded entire function on $\comp^n$, hence a polynomial, see, e.g.\ \cite[Cor.~1.7]{LG86}. The fact that $f$ is a polynomial in this case was established in \cite{BSS23} and \cite{GK23}.
\end{remark}

\begin{remark}\label{trivial}
Let $n = 2$, $g(x) := (1 + |x|)^k$, $k \in \nat$, $f(x_1, x_2) := (x_1 + ix_2)^k$ (or $f(x_1, x_2) := (x_1 + ix_2)^k + (x_1 - ix_2)^k$ if one prefers to have a real-valued $f$). Then $f \in L^\infty_{g^{-1}}(\real^2)$, $\Delta f = 0$, $f(x + iy_1 \mathbf{e}_1)  = (x_1 + iy_1 + ix_2)^k$ for any $y_1 \in \real$, see \eqref{ej}, and
\begin{gather*}
    \frac{\|f(\cdot + iy_1 \mathbf{e}_1)\|_{L_{g^{-1}}^\infty(\real^2)}}{g(y_1 \mathbf{e}_1)}
    \ge \frac{|y_1|^k}{(1 + |y_1|)^k} \xrightarrow[|y_1| \to \infty]{} 1
    = \|f\|_{L_{g^{-1}}^\infty(\real^2)}.
\end{gather*}
So,  the factor $e^{S_g(|y|)|y|} \le C_k\, g(y)$, see \eqref{gpol}, is optimal in \eqref{Hentest0} in this case.

\bigskip
The case $g(x)  = e^{a |x|^b}$, $a > 0$, $b \in [0, 1)$, is perhaps more interesting. Let us take $b = \frac12$. Then it follows from \eqref{gab} that
$e^{S_g(|y|)|y|} = (g(y))^{\sqrt{2}}$.  Let us show that one cannot replace this factor in \eqref{Hentest0} with $(g(y))^{\sqrt{2}\,(1 - \varepsilon)}$, $\varepsilon > 0$. Take any $\varepsilon > 0$. Since
\begin{gather*}
    \sqrt[4]{1 + \tau^2}\, \cos\left(\frac12\arctan{\frac1\tau}\right)
    \xrightarrow[\tau \to 0, \; \tau > 0]{} \frac{1}{\sqrt{2}},
\intertext{there exists some $\tau_\varepsilon > 0$ such that}
    \sqrt[4]{1 + \tau_\varepsilon^2}\, \cos\left(\frac12\arctan{\frac1\tau_\varepsilon}\right)
    \le  \frac{1 + \varepsilon}{\sqrt{2}}.
\end{gather*}
Let us estimate $\Re\sqrt{x_1 + i\kappa x_2}$, where $x = (x_1, x_2) \in \real^2$,  $\kappa > 0$ is a constant to be chosen later, and $\sqrt{\cdot}$ is the branch of the square root that is analytic in $\comp\setminus (-\infty, 0]$ and positive on $(0, +\infty)$. If $x_1 \ge \tau_\epsilon \kappa |x_2|$, then
\begin{align*}
    \Re\sqrt{x_1 + i\kappa x_2}
    \le \left|\sqrt{x_1 + i\kappa x_2}\right|
    = \sqrt[4]{x_1^2 + \kappa^2x_2^2}
    &\le \sqrt[4]{\left(1 + \frac{1}{\tau_\varepsilon^2}\right)x_1^2} \\
    &\le \left(1 + \frac{1}{\tau_\varepsilon^2}\right)^{1/4} \sqrt{x_1} \\
    &\le \left(1 + \frac{1}{\tau_\varepsilon^2}\right)^{1/4} \sqrt{|x|}.
\end{align*}
If $0 < x_1 < \tau_\epsilon \kappa |x_2|$, then
\begin{align*}
    \Re\sqrt{x_1 + i\kappa x_2}
    & =  \left|\sqrt{x_1 + i\kappa x_2}\right| \cos\left(\frac12\arctan{\frac{\kappa |x_2|}{x_1}}\right) \\
    &\le \left|\sqrt{\tau_\epsilon \kappa |x_2| + i\kappa x_2}\right| \cos\left(\frac12\arctan{\frac1\tau_\varepsilon}\right) \\
    &= \kappa^{1/2} |x_2|^{1/2}\, \sqrt[4]{1 +\tau^2_\varepsilon}\, \cos\left(\frac12\arctan{\frac1\tau_\varepsilon}\right)\\
    &\le \frac{1 + \varepsilon}{\sqrt{2}} \kappa^{1/2} |x|^{1/2}.
\end{align*}
Now, take $\kappa_\varepsilon \ge 1$ such that
\begin{gather*}
    \frac{1 + \varepsilon}{\sqrt{2}} \kappa_\varepsilon^{1/2}
    \ge \left(1 + \frac{1}{\tau_\varepsilon^2}\right)^{1/4} .
\end{gather*}
Then
\begin{equation}\label{sqrt}
    \Re\sqrt{x_1 + i\kappa_\varepsilon x_2}
    \le \frac{1 + \varepsilon}{\sqrt{2}} \kappa_\varepsilon^{1/2} |x|^{1/2}
\end{equation}
for $x_1 > 0$. If $x_1 \le 0$, then the argument of $\sqrt{x_1 + i\kappa_\varepsilon x_2}$ belongs to $\pm [\pi/4, \pi/2]$, depending on the sign of $x_2$. Hence,
\begin{align*}
    \Re\sqrt{x_1 + i\kappa_\varepsilon x_2} \le \left|\sqrt{x_1 + i\kappa_\varepsilon x_2}\right| \cos{\frac{\pi}{4}}
    \le \frac{1}{\sqrt{2}} \kappa_\varepsilon^{1/2} |x|^{1/2} ,
\end{align*}
and \eqref{sqrt} holds for all $x = (x_1, x_2) \in \real^2$.

Since the Taylor series of $\cos{w}$ contains only even powers of $w$, $\cos(i\sqrt{z})$ is an analytic function of $z \in \comp$. So, $\cos(i\sqrt{x_1 + ix_2})$ is a harmonic function of $x = (x_1, x_2) \in \real^2$. Hence $f(x_1, x_2) := \cos(i\sqrt{x_1 + i\kappa_\varepsilon x_2})$ is a solution of the elliptic partial differential equation
\begin{gather*}
    \left(\partial^2_{x_1} + \frac{1}{\kappa_\varepsilon^2}\, \partial^2_{x_2}\right) f(x_1, x_2) = 0 .
\end{gather*}
It follows from \eqref{sqrt} that
\begin{gather*}
    |f(x_1, x_2)|
    \le \frac12\left(1 + e^{\Re\sqrt{x_1 + i\kappa_\varepsilon x_2}}\right)
    \le e^{\frac{1 + \varepsilon}{\sqrt{2}} \kappa_\varepsilon^{1/2} |x|^{1/2}} .
\end{gather*}
So, $f \in L_{g^{-1}}^\infty(\real^2)$, where $g(x)  = e^{a |x|^{1/2}}$ with $a = \frac{1 + \varepsilon}{\sqrt{2}} \kappa_\varepsilon^{1/2}$. Clearly, the analytic continuation of $f$ to $\comp^2$ is given by the formula
\begin{gather*}
    f(x_1 + iy_1, x_2 + iy_2) =  \cos\left(i\sqrt{x_1 + iy_1 + i\kappa_\varepsilon (x_2 + iy_2)}\right) .
\end{gather*}
Finally,  see \eqref{ej},  letting $(-\infty,0) \ni y_2 \to -\infty$, we arrive at 
\begin{align*}
    \frac{\|f(\cdot + iy_2 \mathbf{e}_2)\|_{L_{g^{-1}}^\infty(\real^2)}}{(g(y_2 \mathbf{e}_2))^{\sqrt{2}\,(1 - \varepsilon)}}
    & \ge \frac{|f(0 + iy_2 \mathbf{e}_2)|}{g(0)(g(y_2 \mathbf{e}_2))^{\sqrt{2}\,(1 - \varepsilon)}}
    = \frac{\left|\cos(i\sqrt{-\kappa_\varepsilon y_2})\right|}{e^{\sqrt{2}\,(1 - \varepsilon)\frac{1 + \varepsilon}{\sqrt{2}} \kappa_\varepsilon^{1/2} |y_2|^{1/2}}}
    \ge  \frac{e^{\kappa_\varepsilon^{1/2} |y_2|^{1/2}}}{2 e^{(1 - \varepsilon^2) \kappa_\varepsilon^{1/2} |y_2|^{1/2}}} \\
   &  =  \frac 12\,e^{\varepsilon^2 \kappa_\varepsilon^{1/2} |y_2|^{1/2}}      \xrightarrow[y_2 \to -\infty]{} \infty .
\end{align*}

\end{remark}


\section{Concluding remarks}\label{Rems}

Corollary \ref{th:3.1} shows that sub-exponentially growing solutions of $m(D)f=0$ admit analytic continuation to entire functions on $\comp^n$. It is well known that no growth restrictions are necessary in the case when $m(D)$ is an elliptic partial differential operator with constant coefficients, and every solution of $m(D)f=0$ in $\real^n$ admits analytic continuation to an entire function on $\comp^n$, see \cite{P39,E60}.

\begin{remark}
The latter result has a local version similar to Hayman's theorem on harmonic functions, see  \cite[Thm.~1]{H70}: for every elliptic partial differential operator $m(D)$ with constant coefficients there exists a constant $c_m  \in (0, 1)$ such that every solution of  $m(D)f=0$ in the ball $\{x \in \real^n : \ |x| < R\}$ of any radius $R > 0$ admits continuation to an analytic function in the ball $\{x \in \comp^n : \ |x| < c_m R\}$. Indeed, let $m_0(D) = \sum_{|\alpha| = N} a_\alpha D^\alpha$ be the principal part of $m(D) = \sum_{|\alpha| \le N} a_\alpha D^\alpha$. There exists $C_m > 0$ such that
\begin{gather*}
    \sum_{|\alpha| = N} a_\alpha (a + ib)^\alpha = 0, \quad a, b \in \real^n \quad\implies\quad |a| \ge C_m |b|,
\end{gather*}
see, e.g.\ \cite[\S7]{S59}. Then the same argument as in the proof of \cite[Cor.~8.2]{KL18} shows that $f$ admits continuation to an analytic function in the ball $\left\{x \in \comp^n : \ |x| < (1 + C_m^{-2})^{-1/2} R\right\}$. Note that in the case of the Laplacian, one can take $C_m = 1$ and $c_m = (1 + C_m^{-2})^{-1/2} = \frac{1}{\sqrt{2}}$, which is the optimal constant for harmonic functions, see \cite{H70}.
\end{remark}

Let us return to equations in $\real^n$. Below, $m(\xi)$ will always denote a polynomial with $\left\{\xi \in \real^n \mid  m(\xi) = 0\right\}  \subseteq \{0\}$. For non-elliptic partial differential operators $m(D)$, one needs to place growth restrictions on solutions of $m(D)f=0$ to make sure that they
admit analytic continuation to entire functions on $\comp^n$.

We say that a function $f$ defined on $\real^n$ (or $\comp^n$) is of \emph{infra-exponential} growth, if for every $\varepsilon > 0$, there exists $C_\varepsilon > 0$ such that
\begin{gather*}
    |f(z)| \le C_\varepsilon e^{\varepsilon |z|}
    \quad\text{for all}\quad z \in \real^n \ (z \in \comp^n).
\end{gather*}

Let $\mu :  [0, \infty) \to [0, \infty)$ be an increasing function, which increases to infinity and satisfies
\begin{gather*}
    \mu(t) \le At + B , \quad t \ge 0
\end{gather*}
for some $A, B > 0$, and
\begin{equation}\label{muBD}
    \int_1^\infty \frac{\mu(t)}{t^2}\, dt < \infty .
\end{equation}
Suppose $\left\{\xi \in \real^n \mid  m(\xi) = 0\right\}  = \{0\}$. Then, under additional restrictions on $\mu$, every solution $f$ of $m(D)f=0$ that has growth $O(e^{\varepsilon \mu(|x|)})$ for every $\varepsilon > 0$ admits analytic continuation to an entire function of infra-exponential growth on $\comp^n$, see \cite{Kan00}. It is easy to see that \eqref{muBD} is equivalent to the Beurling--Domar condition \eqref{B-D} for
 $g(x) := e^{\varepsilon \mu(|x|)}$.

One cannot replace $O(e^{\varepsilon \mu(|x|)})$ with $O(e^{\varepsilon |x|})$ in the above result without placing a restriction on the complex zeros of $m$. If there exists $\delta > 0$ such that $m(\zeta)$ has no complex zeros in
\begin{equation}\label{delta}
    |\Im \zeta| < \delta , \quad |\Re \zeta| > \delta^{-1} ,
\end{equation}
then every solution of $m(D)f=0$ that, together with its partial derivatives up to the order of $m(D)$, is of infra-exponential growth on $\real^n$, admits analytic continuation to an entire function of infra-exponential growth on $\comp^n$, see \cite{Kan98,Kan00}.

On the other hand, if for every $\delta > 0$, \eqref{delta} contains complex zeros of $m(\zeta)$, then $m(D)f=0$ has a solution in $C^\infty$ all of whose derivatives are of infra-exponential growth
 on $\real^n$, but which is not entire infra-exponential in $\comp^n$. The proof of the latter result in \cite{Kan98,Kan00} is not constructive, and the author writes: ``\emph{Unfortunately we cannot present concrete examples of such solutions}''; however, it is not difficult to construct, for any $\varepsilon > 0$, a solution in $C^\infty$ all of whose derivatives have growth $O(e^{\varepsilon |x|})$, but which is not real-analytic. Indeed, according to the assumption, there exist complex zeros
\begin{gather*}
    \zeta_k = \xi_k + i\eta_k ,
    \quad  \xi_k , \eta_k \in \real^n , \qquad k \in \nat
\end{gather*}
of $m(\zeta)$ such that
\begin{equation}\label{etak}
    |\eta_k| < k^{-1} , \quad |\xi_k| > k .
\end{equation}
Choosing a subsequence, we can assume that $\omega_k := |\xi_k|^{-1} \xi_k$ converge to a point $\omega_0 \in \Ss^{n - 1} := \{\xi \in \real^n : \, |\xi| = 1\}$ as $k \to \infty$, and that $|\omega_k - \omega_0| < 1$ for all $k \in \nat$. Then
\begin{equation}\label{cos}
    \omega_k \cdot \omega_0
    = \frac{|\omega_k|^2 + |\omega_0|^2 - |\omega_k - \omega_0|^2}{2}
    > \frac{1 + 1 - 1}{2} = \frac12, \quad k \in \nat .
\end{equation}
Consider
\begin{equation}\label{f1}
    f(x)
    := \sum_{k > \varepsilon^{-1}}  \frac{e^{i\zeta_k\cdot x}}{e^{|\xi_k|^{1/2}}}
    = \sum_{k > \varepsilon^{-1}}  \frac{e^{i\xi_k\cdot x - \eta_k\cdot x}}{e^{|\xi_k|^{1/2}}},
    \quad x \in \real^n .
\end{equation}
Then, for every multi-index $\alpha$ and  every  $x\in\real^n$,
\begin{align*}
    |\partial^\alpha f(x)|
    = \left|\sum_{k > \varepsilon^{-1}}  \frac{(i\zeta_k)^\alpha e^{i\zeta_k\cdot x}}{e^{|\xi_k|^{1/2}}}\right|
    &\le \sum_{k > \varepsilon^{-1}} \frac{(|\xi_k| + 1)^{|\alpha|} e^{|\eta_k| |x|}}{e^{|\xi_k|^{1/2}}} \\
    &\le e^{\varepsilon |x|} \sum_{k > \varepsilon^{-1}} \frac{(|\xi_k| + 1)^{|\alpha|}}{e^{|\xi_k|^{1/2}}} =: C_\alpha e^{\varepsilon |x|} ,
\end{align*}
see \eqref{etak}. Further,
\begin{gather*}
    m(D) f(x) = \sum_{k > \varepsilon^{-1}}  \frac{m(\zeta_k) e^{i\zeta_k\cdot x}}{e^{|\xi_k|^{1/2}}} = 0 .
\end{gather*}
On the other hand, $f$ is not real-analytic. Before we prove this, note that formally putting $x - i t \omega_0$, $t > 0$ in place of $x$ in the right-hand side of \eqref{f1}, one gets a divergent series. Indeed, its terms can be estimated as follows
\begin{align*}
    \left|\frac{e^{i\xi_k\cdot x + t \xi_k\cdot \omega_0 - \eta_k\cdot x + it\eta_k\cdot \omega_0}}{e^{|\xi_k|^{1/2}}}\right|
    = \frac{e^{t |\xi_k| \omega_k\cdot \omega_0 - \eta_k\cdot x}}{e^{|\xi_k|^{1/2}}}
    \ge e^{-\varepsilon |x|} \frac{e^{t |\xi_k|/2}}{e^{|\xi_k|^{1/2}}} \to \infty
\end{align*}
as $k \to \infty$, see \eqref{etak}, \eqref{cos}.

For any $j > \varepsilon^{-1}$, there exists $\ell_j \in \nat$ such that
\begin{equation}\label{ell}
    \ell_j \le |\xi_j|^{1/2} < \ell_j + 1 .
\end{equation}
It is clear that $\ell_j \to \infty$ as $j \to \infty$, see \eqref{etak}. Note that
\begin{gather*}
    |\arg \left(\omega_0\cdot\zeta_k\right)|
    \le \frac{|\omega_0\cdot\eta_k|}{|\omega_0\cdot\xi_k|} \le \frac{2}{k |\xi_k|}.
\end{gather*}
If $|\xi_k| \ge {6\ell_j}/{\pi k}$, then
\begin{gather*}
    |\arg \left(\omega_0\cdot\zeta_k\right)^{\ell_j}|
    \le \frac{2\ell_j}{k |\xi_k|} \le \frac{\pi}{3},
\intertext{and}
    \Re \left(\omega_0\cdot\zeta_k\right)^{\ell_j}
    \ge \frac12 \left|\omega_0\cdot\zeta_k\right|^{\ell_j}
    \ge \frac1{2^{\ell_j + 1}} |\xi_k|^{\ell_j} .
\end{gather*}
Clearly, $|\xi_j| \ge \frac{6\ell_j}{\pi j}$ for sufficiently large $j$, see \eqref{ell}. Hence, one has the following estimate for the directional derivative $\partial_{\omega_0}$
\begin{align*}
    \left|\left((-i\partial_{\omega_0})^{\ell_j}f\right)(0)\right|
    &\ge \sum_{k > \varepsilon^{-1}}  \frac{\Re \left(\omega_0\cdot\zeta_k\right)^{\ell_j}}{e^{|\xi_k|^{1/2}}} \\
    &\ge  -\sum_{k > \varepsilon^{-1}, \ |\xi_k| < \frac{6\ell_j}{\pi k}}  \frac{|\zeta_k|^{\ell_j}}{e^{|\xi_k|^{1/2}}}
            + \sum_{k > \varepsilon^{-1}, \ |\xi_k| \ge \frac{6\ell_j}{\pi k}} \frac{|\xi_k|^{\ell_j}}{2^{\ell_j + 1}e^{|\xi_k|^{1/2}}} \\
    &\ge  -\sum_{k > \varepsilon^{-1}, \ |\xi_k| < \frac{6\ell_j}{\pi k}}  \frac{\left(|\xi_k| + \frac1k\right)^{\ell_j}}{e^{|\xi_k|^{1/2}}}
            + \frac{|\xi_j|^{\ell_j}}{2^{\ell_j + 1}e^{|\xi_j|^{1/2}}} \\
    &\ge  -\sum_{k > \varepsilon^{-1}, \ |\xi_k| < \frac{6\ell_j}{\pi k}}  \frac{1}{e^{|\xi_k|^{1/2}}}\left(\frac{10\ell_j}{\pi k}\right)^{\ell_j}
            + \frac{\ell_j^{2\ell_j}}{2^{\ell_j + 1}e^{(\ell_j^2 + 1)^{1/2}}} \\
    &\ge  - (10\ell_j)^{\ell_j}\sum_{k = 1}^\infty  \frac{1}{e^{|\xi_k|^{1/2}} k^2}\
            + \frac{\ell_j^{2\ell_j}}{2^{\ell_j + 1}e^{\ell_j + 1}}\\
    &= - C(10\ell_j)^{\ell_j} + (2e)^{-(\ell_j + 1)} \ell_j^{2\ell_j} .
\end{align*}
Hence,
\begin{gather*}
    \left|\left((-i\partial_{\omega_0})^{\ell_j}f\right)(0)\right|
    \ge \ell_j^{\frac{3}{2}\,\ell_j}
\end{gather*}
for all sufficiently large $j$, which means that $f$ is not real-analytic in a neighbourhood of $0$.

The operator $m(D)$ in the previous example is not hypoelliptic. If $m(D)$ is hypoelliptic, then every solution of $m(D) f = 0$, such that $|f(x)| \le A e^{a|x|}$, $x \in \real^n$, for some constants $A, a > 0$, admits analytic continuation to an entire function of order one on $\comp^n$, see \cite[\S 4, Cor.~2]{G66}. For elliptic operators, this result can be strengthened: every solution of $m(D) f = 0$, such that $|f(x)| \le A e^{a|x|^\beta}$, $x \in \real^n$, for $\beta \ge 1$ and some constants $A, a > 0$, admits analytic continuation to an entire function of order $\beta$ on $\comp^n$, see \cite[\S 4, Cor.~3]{G66}. Let us show that for every $\beta > 1$ there exists a semi-elliptic operator $m(D)$, see \cite[Thm.~11.1.11]{H83_2}, and a $C^\infty$ solution of $m(D) f = 0$, all of whose derivatives have growth $O(e^{a|x|^\beta})$, but which does not admit analytic continuation to an entire function on $\comp^n$.

A simple example of such a semi-elliptic operator is $\partial_{x_1}^2 + \partial_{x_2}^{4\ell + 2}$ with $\ell \in \nat$ satisfying $1 + \frac{1}{2\ell} \le \beta$, i.e.\ $\ell \ge \frac{1}{2(\beta - 1)}$.

Let
\begin{gather*}
    f(x_1, x_2)
    := \sum_{k = 1}^\infty \frac{e^{-ik^{2\ell + 1}x_1 + kx_2}}{e^{k^{2\ell + 1}}},
    \quad (x_1, x_2) \in \real^2 .
\end{gather*}
If $x_2 > 0$, then the function $t \mapsto tx_2 - t^{2\ell + 1}$ achieves a maximum at $t = \left(\frac{x_2}{2\ell + 1}\right)^{\frac{1}{2\ell}}$, and this maximum is equal to
\begin{gather*}
    2\ell \left(\frac{1}{2\ell + 1}\right)^{1 + \frac{1}{2\ell}}x_2^{1 + \frac{1}{2\ell}}
    =: c_\ell\, x_2^{1 + \frac{1}{2\ell}} .
\end{gather*}
Hence, for every multi-index $\alpha$,
\begin{align*}
    |\partial^\alpha f(x_1, x_2)|
    &\le \sum_{k = 1}^\infty k^{(2\ell + 1)|\alpha|} e^{kx_2 - k^{2\ell + 1}} \\
    &= \sum_{k = 1}^{\left[x_2^{\frac{1}{2\ell}}\right] + 1} k^{(2\ell + 1)|\alpha|} e^{kx_2 - k^{2\ell + 1}}
        + \sum_{k = {\left[x_2^{\frac{1}{2\ell}}\right]} + 2}^\infty k^{(2\ell + 1)|\alpha|} e^{k\left(x_2 - k^{2\ell}\right)} \\
    &\le \left(\left[x_2^{\frac{1}{2\ell}}\right] + 1\right)^{(2\ell + 1)|\alpha| + 1}  e^{c_\ell\, x_2^{1 + \frac{1}{2\ell}}}
        + \sum_{k = 1}^\infty k^{(2\ell + 1)|\alpha|} e^{-k} \\
    &\le 2^{(2\ell + 1)|\alpha| + 1}\left(x_2^{2|\alpha| +1} + 1\right) e^{c_\ell\, x_2^{1 + \frac{1}{2\ell}}}
        + c_{\ell, \alpha}\\
    &\le C_{\ell, \alpha} e^{(c_\ell + 1) x_2^{1 + \frac{1}{2\ell}}}.
\end{align*}
If $x_2 \le 0$, then
\begin{gather*}
    |\partial^\alpha f(x_1, x_2)|
    \le \sum_{k = 1}^\infty \frac{k^{(2\ell + 1)|\alpha|}}{e^{k^{2\ell + 1}}}
    < \sum_{j = 1}^\infty \frac{j^{|\alpha|}}{e^{j}}
    =: C_\alpha < \infty .
\end{gather*}
So, $f \in C^\infty(\real^2)$, and
$\partial^\alpha f(x_1, x_2) = O\left(e^{(c_\ell + 1) |x_2|^{1 + \frac{1}{2\ell}}}\right) = O\left(e^{(c_\ell + 1) |x|^{1 + \frac{1}{2\ell}}}\right)$.
It is easy to see that $\left(\partial_{x_1}^2 + \partial_{x_2}^{4\ell + 2}\right)f(x_1, x_2) = 0$.

The function $f$ admits analytic continuation to the set
\begin{gather*}
    \Pi_1 := \left\{(z_1, z_2) \in \comp^2 | \ \Im z_1 < 1\right\} .
\end{gather*}
Indeed, let
\begin{align*}
    f(z_1, z_2)
    = f(x_1 + iy_1, x_2 + iy_2)
    &= \sum_{k = 1}^\infty \frac{e^{-ik^{2\ell + 1}(x_1 + iy_1) + k(x_2 + iy_2)}}{e^{k^{2\ell + 1}}} \\
    &= \sum_{k = 1}^\infty e^{i\left(k y_2 - k^{2\ell + 1}x_1\right)} e^{k^{2\ell + 1}(y_1 - 1) + kx_2} .
\end{align*}
It is easy to see that the last series is uniformly convergent on compact subsets of $\Pi_1$. So, $f$ admits analytic continuation to $\Pi_1$. On the other hand, $f(iy_1, 0) \to \infty$ as $y_1 \to 1 - 0$. Indeed,
\begin{gather*}
    f(iy_1, 0) = \sum_{k = 1}^\infty  e^{k^{2\ell + 1}(y_1 - 1)} .
\end{gather*}
Take any $N \in \nat$. If $y_1 > 1 - N^{-(2\ell + 1)}$, then
\begin{gather*}
    f(iy_1, 0)
    >  \sum_{k = 1}^\infty  e^{-k^{2\ell + 1}N^{-(2\ell + 1)}}
    >  \sum_{k = 1}^N  e^{-k^{2\ell + 1}N^{-(2\ell + 1)}}
    \ge \sum_{k = 1}^N  e^{-1}
    = \frac{N}{e} .
\end{gather*}
So, $f(iy_1, 0) \to \infty$ as $y_1 \to 1 - 0$.


\begin{ack}
    We are grateful for the careful comments of an expert referee, who helped us to improve the exposition of this paper. Financial support for the first two authors through the DFG-NCN Beethoven Classic 3 project SCHI419/11-1 \& NCN 2018/31/G/ST1/02252, the  6G-life project (BMBF programme ``Souver\"an. Digital. Vernetzt.'' 16KISK001K) and the SCADS.AI centre is gratefully acknowledged.
\end{ack}



\frenchspacing


\frenchspacing
\begin{thebibliography}{00}
\bibitem{Ali-et-al20}
N. Alibaud, F. del Teso, J. Endal, and  E.R. Jakobsen,
The Liouville theorem and linear operators satisfying the maximum principle.
\emph{Journal des Mathematiques Pures et Appliqu\'ees} \textbf{142}, 229--242, 2020.

\bibitem{BS22}
D. Berger and R.L. Schilling,
On the Liouville and strong Liouville properties for a class of non-local operators.
\emph{Mathematica Scandinavica} \textbf{128}, 365--388, 2022.


\bibitem{BSS23}
D. Berger, R.L. Schilling, and E.~Shargorodsky, The Liouville theorem for a class of Fourier multipliers and its connection to coupling.
(arXiv:2211.08929, submitted).

\bibitem{C78}
J.B. Conway, \emph{Functions of one complex variable I}. Springer, New York, 1978.

\bibitem{D56}
Y. Domar,
Harmonic analysis based on certain commutative Banach algebras.
\emph{Acta Mathematica} \textbf{96}, 1--66,  1956.

\bibitem{E60}
L. Ehrenpreis,
Solution of some problems of division. Part IV: Invertible and elliptic operators.
\emph{American Journal of Mathematics} \textbf{82}, 522--588, 1960.

\bibitem{F63}
A. Friedman,
\emph{Generalized functions and partial differential equations.}
Prentice-Hall, Englewood Cliffs (NJ), 1963.

\bibitem{G06}
J.B. Garnett,
\emph{Bounded analytic functions.}  Springer, New York, 2006.


\bibitem{gel-vil64}
I.M. Gelfand and N.Ya. Vilenkin,
\emph{Generalized functions. Vol.\ 4: Applications of harmonic analysis}. Academic Press, New York, 1964.


\bibitem{G07}
K. Gr\"ochenig,
Weight functions in time-frequency analysis. In:
L. Rodino (ed.) et al., \emph{Pseudo-differential operators. Partial differential equations and time-frequency analysis}.  Fields Institute Communications
\textbf{52}, 343--366, 2007.

\bibitem{G66}
V.V. Gru\v{s}in,
The connection between local and global properties of the solutions of hypo-elliptic equations with constant coefficients.
\emph{Matematicheskii Sbornik \textup{(}N.S.\textup{)}} \textbf{66(108)}, 525--550,  1965 (Russian).

\bibitem{GK23}
T. Grzywny and M. Kwa\'{s}nicki, Liouville's theorems for L\'{e}vy operators.
(arXiv:2301.08540).


\bibitem{H70}
W.K. Hayman,
Power series expansions for harmonic functions.
\emph{Bulletin of the London Mathematical Society} \textbf{2}, 152--158, 1970.


\bibitem{HP57}
E. Hille and R.S. Phillips,
\emph{Functional analysis and semigroups}.
American Mathematical Society, Providence (RI), 1957.


\bibitem{H83}
L. H\"ormander,
\emph{The analysis of linear partial differential operators. I: Distribution theory and Fourier analysis}.
Springer, Berlin, 1983.

\bibitem{H83_2}
L. H\"ormander,
\emph{The analysis of linear partial differential operators. II: Differential operators with constant coefficients}.
Springer, Berlin, 1983.


\bibitem{Kan98}
A. Kaneko,
Liouville type theorem for solutions of infra-exponential growth of linear partial differential equations with constant coefficients.
\emph{Natural science report of the Ochanomizu University} \textbf{49}, 1, 1--5, 1998.

\bibitem{Kan00}
A. Kaneko,
Liouville type theorem for solutions of linear partial differential equations with constant coefficients.
\emph{Annales Polonici Mathematici} \textbf{74}, 143--159, 2000.


\bibitem{KL18}
D. Khavinson and E. Lundberg,
\emph{Linear holomorphic partial differential equations and classical potential theory}.
American Mathematical Society, Providence (RI), 2018.


\bibitem{K98}
P. Koosis,
\emph{The logarithmic integral. I}.
Cambridge University Press, Cambridge, 1998.


\bibitem{LG86}
P. Lelong and L. Gruman,
\emph{Entire functions of several complex variables}.
Springer, Berlin, 1986.

\bibitem{L64}
B.Ya. Levin,
\emph{Distribution of zeros of entire functions.}
American Mathematical Society, Providence (RI), 1964.

\bibitem{P39}
I.G. Petrowsky,
Sur l'analyticit\'e des solutions des syst\`{e}mes d'\'equations diff\'erentielles. \emph{Matematicheskii Sbornik \textup{(}N.S.\textup{)}}
\textbf{5(47)}, 1, 3--70, 1939.

\bibitem{Ru73}
W. Rudin,
\emph{Functional analysis}.
McGraw-Hill, New York, 1973.

\bibitem{sato}
K. Sato,
\emph{L\'evy processes and infinitely divisible distributions}.
Cambridge University Press, Cambridge, 2013.

\bibitem{S59}
G.E. \v{S}ilov, Local properties of solutions of partial differential equations with constant coefficients.
\emph{Translations of the American Mathematical Society \textup{(}Ser.~2\textup{)}} \textbf{42}, 129--173, 1964.

\bibitem{S61}
G.E. \v{S}ilov, An analogue of a theorem of Laurent Schwartz.
\emph{Izvesti\^{a} Vys\v{s}ih U\v{c}ebnyh Zavedenij, Matematika} \textbf{1961}, 4, 137--147, 1961.


\bibitem{wendland}
H. Wendland, \emph{Scattered Data Approximation}. Cambridge University Press, Cambridge, 2005.

\end{thebibliography}
\end{document}